\documentclass[reqno,a4paper,10pt]{amsart}
\usepackage{amsfonts}
\usepackage{amssymb,amsmath, amsrefs}
\usepackage{verbatim}
\usepackage{enumerate}

\usepackage[retainorgcmds]{IEEEtrantools}
\theoremstyle{plain}
\newtheorem{theorem}{Theorem}[section]

\newtheorem{lem}[theorem]{Lemma}

\newtheorem{cor}[theorem]{Corollary}

\newtheorem{prop}[theorem]{Proposition}

\newtheorem{maintheorem}{Theorem}
\newtheorem{restatement}{Theorem}
\newtheorem{defn}[theorem]{Definition}

\theoremstyle{definition} \theoremstyle{definition}

\theoremstyle{remark}

\newtheorem{remark}{Remark}

\def\SL{{\rm SL}}

\def\GSp{{\rm GSp}}
\def\PGSp{{\rm PGSp}}
\def\Sp{{\rm Sp}}

\def\GL{{\rm GL}}

\def\GSO{{\rm GSO}}

\def\SO{{\rm SO}}
\def\O{{\rm O}}

\def\tr{{\rm tr\,}}

\begin{document}
\title{On the degree five $L$-function for $\GSp(4)$}
\author{Daniel File}
\vspace{2mm}
\address{Departement of Mathematics, 14 MacLean Hall, Iowa City, Iowa 52242-1419}\email{daniel-file@uiowa.edu}
\thanks{This work was done while I was a gruaduate student at Ohio State University as part of my Ph.D. dissertation. I wish to thank my advisor Jim Cogdell for being a patient teacher and for his helpful discussions about this work. I also thank Ramin Takloo-Bighash for many useful conversations.}
\begin{abstract}
I give a new integral representation for the degree five (standard) $L$-function for automorphic representations of $\GSp(4)$ that is a refinement of integral representation of Piatetski-Shapiro and Rallis. The new integral representation unfolds to produce the Bessel model for $\GSp(4)$ which is a unique model. The local unramified calculation uses an explicit formula for the Bessel model and differs completely from Piatetski-Shapiro and Rallis.
\end{abstract}
 \maketitle
\date{}

\section{Introduction}

In 1978 Andrianov and Kalinin established an integral representation for the degree $2n+1$ standard $L$-function of a Siegel modular form of genus $n$~\cite{andrianovkalinin1978}. Their integral involves a theta function and a Siegel Eisenstein series. The integral representation allowed them to prove the meromorphic continuation of the $L$-function, and in the case when the Siegel modular form has level $1$ they established a functional equation and determined the locations of possible poles. 

Piatetski-Shapiro and Rallis became interested in the construction of Andrianov and Kalinin because it seems to produce Euler products without using any uniqueness property. Previous examples of integral representations used either a unique model such as the Whittaker model, or the uniqueness of the invariant bilinear form between an irreducible representation and its contragradient. It is known that an automorphic representation of $\Sp_4$ (or $\GSp_4$) associated to a Siegel modular form does not have a Whittaker model. Piatetski-Shapiro and Rallis adapted the integral representation of Andrianov and Kalinin to the setting of automorphic representations and were able to obtain Euler products ~\cite{piatetski-shapirorallis1988}; however, the factorization is not the result of a unique model that would explain the local-global structure of Andrianov and Kalinin. They considered the integral
\begin{equation*}
 \int \limits_{\Sp_{2n}(F) \backslash \Sp_{2n}(\mathbb{A})} \phi(g) \theta_T(g) E(s,g) \, dg
\end{equation*}
where $E(s,g)$ is an Eisenstein series induced from a character of the Siegel parabolic subgroup, $\phi$ is a cuspidal automorphic form, $T$ is a $n$-by-$n$ symmetric matrix determining an $n$ dimensional orthogonal space, and $\theta_T(g)$ is the theta kernel for the dual reductive pair $\Sp_{2n} \times \O(V_T)$. 

Upon unfolding their integral produces the expansion of $\phi$ along the abelian unipotent radical $N$ of the Siegel parabolic subgroup. They refer to the terms in this expansion as Fourier coefficients in analogy with the Siegel modular case. The Fourier coefficients are defined as
\begin{equation*}
 \phi_T(g) = \int \limits_{N(F)\backslash N(\mathbb{A})} \phi(ng) \, \psi_T(n) \, dn.
\end{equation*}
Here, $T$ is associated to a character $\psi_T$ of $N(F)\backslash N(\mathbb{A})$. These functions $\phi_T$ do not give a unique model for the automorphic representation to which $\phi$ belongs. The corresponding statement for a finite place $v$ of $F$ is that for a character $\psi_v$ of $N(F_v)$ the inequality
\begin{equation*}
dim_\mathbb{C} \mathrm{Hom}_{N(F_v)}(\pi_v , \psi_v) \leq 1
\end{equation*}
does not hold for all irreducible admissible representation $\pi_v$ of $\Sp_{2n}(F_v)$.

However, Piatetski-Shapiro and Rallis show that their local integral is independent of the choice of Fourier coefficient when $v$ is a finite place and the local representation $\pi_v$ is spherical.
Specifically, they show that for any $\ell_T \in  \mathrm{Hom}_{N(F_v)}(\pi_v , \psi_v)$ the integral
\begin{equation*}
 \int \limits_{Mat_n(\mathcal{O}_v) \cap \GL_n(F_v)} \ell_T \left( \begin{bmatrix} g & \\ & ^{t} g ^{-1} \end{bmatrix} v_0 \right) |\det(g)|_v^{s-1/2} \, dg= d_v(s) L(\pi_v, \frac{2s+1}{2}) \ell_T(v_0)
\end{equation*}
where $v_0$ is the spherical vector for $\pi_v$, $\mathcal{O}_v$ is the ring of integers, and $d_v(s)$ is a product of local $\zeta$-factors.
At the remaining ``bad'' places the integral does not factor, and there is no local integral to compute. However, they showed that the integral over the remaining places is a meromorphic function of $s$.

In this paper I present a new integral representation for the degree five $L$-function for GSp$_4$ which is a refinement of the work of Piatetski-Shapiro and Rallis. Instead of working with the full theta kernel, the construction in this paper uses a theta integral for $\GSp_4 \times \GSO_2$. This difference has the striking effect of producing the Bessel model for $\GSp_4$ and the uniqueness that Piatetski-Shapior and Rallis expected. Therefore, this integral factors as an Euler product over all places. I compute the local unramified integral when the local representation is spherical using the formula due to Sugano~\cite{sugano1985}.

In some instances an integral representation of an $L$-function can be used to prove algebraicity of special values of that $L$-function (up to certain expected transcendental factors). Harris~\cite{harris1981}, Sturm~\cite{sturm1981}, Bocherer~\cite{bocherer1985}, and Panchishkin~\cite{panchishkin1991} applied the integral representation of Andrianov and Kalinin to prove algebraicity of special values of the standard $L$-function of certain Siegel modular forms. Shimura~\cite{shimura2000} also used an integral representation to prove algebraicity of these special values for many Siegel modular forms including forms for every congruence subgroup of $\Sp_{2n}$ over a totally real number field.

Furusawa~\cite{furusawa1993} gave an integral representation for the degree eight $L$-function for $\GSp_4 \times \GL_2$. Furusawa's integral representation unfolds to give the Bessel model for $\GSp_4$ times the Whittaker model for $\GL_2$, and he uses Sugano's formula for the spherical Bessel model to compute the unramified local integral. Let $\Phi$ be a genus $2$ Siegel eigen cusp form of weight $\ell$, and let $\pi=\otimes_v \pi_v$ be the automorphic representation for $\GSp_4$ associated to it. Let $\Psi$ be an elliptic (genus 1) eigen cusp form of weight $\ell$, and let $\tau=\otimes_v \tau_v$ be the associated representation for $\GL_2$. As an application of his integral representation Furusawa proved an algebraicity result for special values of the degree eight $L$-function $L(s, \Phi \times \Psi)$ provided that for all finite places $v$ both $\pi_v$ and $\tau_v$ are spherical. This condition is satisfied when $\Phi$ and $\Psi$ are modular forms for the full modular groups $\Sp_4(\mathbb{Z})$ and $\SL_2(\mathbb{Z})$, respectively.

A recent result of Saha~\cite{saha2009} includes the explicit computation of Bessel functions of local representations that are Steinberg. This allowed Saha to extend the special value result of Furusawa to the case when $\pi_p$ is Steinberg at some prime $p$. Pitale and Schmidt~\cite{pitaleschmidt2009} considered the local integral of Furusawa for a large class of representations $\tau_p$ and as an application extended the algebraicity result of Furusawa further.

In principle one could explicitly compute the local integral given in this paper using the formula of Saha at a place where the local representation is Steinberg. Considering the algebraicity results of Harris~\cite{harris1981}, Sturm~\cite{sturm1981}, Bocherer~\cite{bocherer1985}, and Panchishkin~\cite{panchishkin1991} that involve the integral of Andrianov and Kalinin, and the explicit computations for Bessel models due to Sugano~\cite{sugano1985}, Furusawa~\cite{furusawa1993}, and Saha~\cite{saha2009}, it would be interesting to see if the integral representation of this paper can be used to obtain any new algebraicity results. This is a question I intend to address in a later work.

\section{Summary of Results}
Let $\pi$ be an automorphic representation of GSp$_4(\mathbb{A})$, $\phi \in V_\pi$, $\nu$ an automorphic character on GSO$_2(\mathbb{A})$ the similitude orthogonal group that preserves the symmetric form determined by the symmetric matrix $T$, $\theta_\varphi(\nu^{-1})$ the theta lift of $\nu^{-1}$ to GSp$_4$ with respect to a Schwartz-Bruhat function $\varphi$,  and $E(s, f, g)$ a Siegel Eisenstein series for a section $f(s, -) \in \text{Ind}_{P(\mathbb{A})}^{G(\mathbb{A})}( \delta_P ^{1/3(s-1/2)})$. Consider the global integral
\begin{equation*}
 I(s;f, \phi, T, \nu, \varphi)=I(s):=\int \limits_{Z_\mathbb{A} GSp_4(F) \backslash GSp_4(\mathbb{A})}
E(s, f, g) \phi(g) \theta_\varphi(\nu^{-1})(g) \, dg.
\end{equation*}

Section~\ref{global} contains the proof that $I(s)$ has an Euler product expansion
\begin{equation*}
 I(s)=\int \limits_{N(\mathbb{A}_\infty) \backslash G_1(\mathbb{A}_\infty)} f(s,g) \phi^{T, \nu}(g) \omega(g, 1) \varphi(1_2) \, dg \cdot \prod \limits_{v < \infty} I_v(s)
\end{equation*}
where integrals $I_v(s)$ are defined to be
\begin{equation*}
I_v(s)=\int \limits_{N(F_v) \backslash G_1(F_v)} f_v(s, g_v) \, \phi_v ^{T, \nu}(g_v) \, \omega_v(g_v, 1) \varphi_v(1_2) \, dg_v.
\end{equation*}
The function $\phi_v^{T,\nu}$ belongs to the Bessel model of $\pi_v$.

Section~\ref{unramifiedchapter} includes the proof that under certain conditions that hold for all but a finite number of places $v$, there is a normalization $I_v^*(s)=\zeta_v(s+1)\zeta_v(2s) \, I_v(s)$ such that
\begin{equation*}
 I_v^*(s)=L(s, \pi_v \otimes \chi_T)
\end{equation*}
where $\chi_T$ is a quadratic character associated to the matrix $T$. Section~\ref{ramified} deals with the finite places that are not covered in Section~\ref{unramifiedchapter}. For these places there is a choice of data so that $I_v(s)=1$. Section~\ref{archimedean} deals with the archimedean places and shows that there choice of data to control the analytic properties of $I_v(s)$.

Combining these analyses give the following theorem.
\begin{maintheorem} \label{maintheorem}
 Let $\pi$ be a cuspidal automorphic representation of GSp$_4(\mathbb{A})$, and $\phi \in V_\pi$. Let $T$ and $\nu$ be such that $\phi^{T,\nu}\neq0$. There exists a choice of section $f(s, -) \in \text{Ind}_{P(\mathbb{A})}^{G(\mathbb{A})}( \delta_P ^{1/3(s-1/2)})$, and some $\varphi=\otimes_v \varphi_v \in \mathcal{S}( \mathbb{X}(\mathbb{A}))$ such that the normalized integral
\begin{equation*}
 I^*(s;f, \phi, T, \nu, \varphi)= d(s) \cdot L^{S}(s, \pi \otimes \chi_{T,v})
\end{equation*}
where $S$ is a finite set of bad places including all the archimedean places. Furthermore, for any complex number $s_0$, there is a choice of data so that $d(s)$ is holomorphic at $s_0$, and $d(s_0) \neq 0$.
\end{maintheorem}

\section{Notation}
Let $F$ be a number field, and let $\mathbb{A}=\mathbb{A}_F$ be its ring of adeles. For a place $v$ of $F$ denote by $F_v$ the completion of $F$ at $v$. For a non-archimedean place $v$ let $\mathcal{O}_v$ be the ring of integers of $F_v$, and let $\mathfrak{p}_v$ be its maximal ideal. Let $q_v=[ \mathcal{O}_v : \mathfrak{p}_v ]$. Let $\varpi_v$ be a choice of uniformizer for $\mathfrak{p}_v$, and let $| \cdot |_v$ be the absolute value on $F_v$, normalized so that $|\varpi_v |_v=q_v^{-1}$.

For a finite set of places $S$, let $\mathbb{A}^S= {\prod \limits_{v \notin S}}^\prime F_v$, and $\mathbb{A}_S={\prod \limits_{v \in S}} F_v$. In particular, $\mathbb{A}_\infty = {\prod \limits_{v | \infty}} F_v$, and $\mathbb{A}_{\text{fin}}={\prod \limits_{v < \infty}}^\prime F_v$.

Denote by $\text{Mat}_n$ the variety of $n \times n$ matrices defined over $F$. $\text{Sym}_n$ is the variety of symmetric $n \times n$ matrices defined over $F$.

Let 
$G= \mathrm{GSp}_4=\{g \in GL_4 \big| \ ^tg J g= \lambda_G(g) J\}$
where
\begin{align*}
J=\begin{bmatrix} & & 1 & \\ & & & 1\\ -1 & & &\\ & -1 & &\end{bmatrix}.
\end{align*}

Fix a maximal compact subgroup $K$ of $G(\mathbb{A})$ such that $K=\prod_{v} K_v$ where $K_v$ is a maximal compact subgroup of $G(F_v)$, and at all but finitely many finite places $K_v =G(\mathcal{O}_v)$. According to~\cite{moeglinwaldspurger1995}*{I.1.4} the subgroups $K_v$ can be chosen so that for every standard parabolic subgroup $P$, $G(\mathbb{A})=P(\mathbb{A})K$, and $M(\mathbb{A}) \cap K$ is a maximal compact subgroup of $M(\mathbb{A})$.

\section{Orthogonal Similitude Groups} \label{orthog}
A matrix $T \in \text{Sym}_2(F)$ with $\det(T) \neq 0$ determines a non-degenerate symmetric bilinear form $( \ , \ )_T$ on an $V_T=F^2$:
\begin{equation*}
(v_1 , v_2)_T:= {^t}v_1 T v_2.
\end{equation*}

The orthogonal group associated to this form (and matrix $T$) is
\begin{equation*}
O(V_T)=\{h\in GL_2 \big| {^t}h T h= T\}.
\end{equation*}

Similarly, the similitude group $GO(V_T) = \{h\in GL_n \big| {^t}h T h=\lambda_T(h) T\}$, and $GSO(V_T)$ is defined to be the Zariski connected component of $GO(V_T)$. Note that since $dim(V_T)=2$, and $h\in GSO(V_T)$, then $\lambda(h) = \det(h)$.

Let $\chi_T$ be the quadratic character associated to $V_T$. If $E/F$ is the discriminant field of $V_T$, i.e. $E=F \left(\sqrt{-\det(T)}\right)$, then
\begin{equation*}
\chi_T : F^\times \backslash \mathbb{A}^\times \rightarrow \mathbb{C}
\end{equation*}
is the idele class character associated to $E$ by class field theory. It has the property that $\chi_T=\otimes \chi_{T,v}$ where $\chi_{T,v}(a)= ( a , - \det(T) )_v$, and $(\, , \, )_v$ denotes the local Hilbert symbol~\cite{soudry1988}*{$\S$ 0.3}. Consequently, each $\chi_{T.v} \circ N_{E_v/F_v} \equiv 1$ where $N_{E_v / F_v}$ is the norm map~\cite{serre1973}*{Chapter III, Proposition 1}. Note that $N_{E_v / F_v}=\det=\lambda_T$.

\subsection{The Siegel Parabolic Subgroup}

Let $P=MN$ be the Siegel parabolic subgroup of $G$, i.e. $P$ stabilizes a maximal isotropic subspace $X=\mathrm{span}_F \{e_1, e_2\}$ where $e_i$ is the ith standard basis vector. Then $P$ has Levi factor $M \cong \mathrm{GL}_1 \times \mathrm{GL}_2$ and unipotent radical $N \cong \text{Sym}_2 \cong \mathbb{G}_a ^3$. For $g \in GL_2$, define
\begin{equation*}
m(g)=
\begin{bmatrix}
g &\\
 & ^t g^{-1}
\end{bmatrix} \in M.
\end{equation*}
For $X \in \text{Sym}_2$, define
\begin{align*}
n(X)=\begin{bmatrix} I_2 & X \\ & I_2 \end{bmatrix} \in N.
\end{align*}
Let $\delta_P$ be the modular character of $P$.

For $m=\begin{bmatrix} g & \\ & ^t g^{-1} \lambda \end{bmatrix} \in M$ and $n \in N$, 
$\delta_P(mn)= | \det(g)^3 \cdot \lambda^{-3}|_{\mathbb{A}} \label{adeq}$.
It is possible to extend
$\delta_P$ to all of $G$. For $g=nmk$ where $n \in N$, $m \in M$, and $k \in K$, define $\delta_P(g)=\delta_P(m)$. This is well defined because $\delta_P(m)=1$ for $m \in M \cap K$.

\section{Bessel Models and Coefficients} \label{bessel1}
\subsection{The Bessel Subgroup}
Let $\psi$ be an additive character of $F \backslash \mathbb{A}$.
There is a bijection between $\text{Sym}_2(F)$ and the characters of $N(F) \backslash N(\mathbb{A})$. For $T \in \text{Sym}_2(F)$ define
\begin{align}
\psi_T : N(F) \backslash N(\mathbb{A}) \rightarrow \mathbb{C} \nonumber \\
\psi_T(n(X))=\psi(tr(TX)). \label{char}
\end{align}

Since $M(F)$ acts on $N(F) \backslash N(\mathbb{A})$, it also acts on its characters. Define $H_T$ to be the connected component of the stabilizer of $\psi_T$ in $M$. 
For $g \in GL_2$, define
\begin{align*}
b(g)=\begin{bmatrix} g & \\ &det(g)\cdot \ {^t}g^{-1}\end{bmatrix}.
\end{align*}
Then
\begin{align*}
H_T &=\left\{ b(g)  \Big| \ {^t}gTg=det(g) \cdot T \right\}.
\end{align*}

Then $H_T$ is an algebraic group defined over $F$ isomorphic to $GSO(V_T)$ where $V_T$ is defined as above. 

The adjoint action of $M(F)$ on the characters of $N(F) \backslash N(\mathbb{A})$ has two types of orbits. They are represented by matrices
\begin{align*}
 T_\rho= \begin{bmatrix} 1 & \\ & -\rho \end{bmatrix} \quad \text{with} \, \rho \notin F^{\times, 2}, \, \text{and} \quad T_{\text{split}}=\begin{bmatrix}  & 1 \\1 & \end{bmatrix}.
\end{align*}
The quadratic spaces corresponding to these matrices have similitude orthogonal groups
\begin{align*}
GSO(V_{T_\rho}) = \left\{ \begin{bmatrix} x & \rho y \\ y & x \end{bmatrix} \Bigg| x^2 - \rho y^2 \neq 0 \right\}.
\end{align*}
and
\begin{align}
GSO(V_{T_{\text{split}}}) = \left\{ \begin{bmatrix} x & \\  & y \end{bmatrix} \Bigg| xy \neq 0 \right\}. \label{gsosplit}
\end{align}

For the rest of this article assume that $\rho \notin F^{\times, 2}$, and only consider $T=T_\rho$.
Define the Bessel subgroup $R=R_T=H_T N$.
Consider a character
\begin{equation*}
\nu : H_T(F) \backslash H_T(\mathbb{A}) \rightarrow \mathbb{C}.
\end{equation*}

Then define
\begin{align*}
\nu \otimes& \psi_T : R(F) \backslash R(\mathbb{A}) \rightarrow \mathbb{C}&\\
\nu \otimes& \psi_T (tn)= \nu (t) \psi_T(n) & t \in H_T(\mathbb{A}), \  n\in N(\mathbb{A}).
\end{align*}
This is well defined since $H_T$ normalizes $\psi_T$.

Similarly, for a place $v$ of $F$ there are local characters
\begin{equation*}
\nu_v \otimes \psi_{T,v} : R(F_v) \rightarrow \mathbb{C}.
\end{equation*}

\subsection{Non-Archimedean Local Bessel Models}
Let $v$ be a finite place of $F$.
Let $\mathcal{B}$ be the space of locally constant functions $\phi : G(F_v) \rightarrow \mathbb{C}$ satisfying
\begin{align*}
 \phi(rg)=\nu_v \otimes \psi_{T,v}(r) \phi(g)
\end{align*}
for all $r \in R(F_v)$ and all $g \in G(F_v)$.

Let $\pi_v$ be an irreducible admissable representation of $G(F_v)$. Piatetski-Shapiro and Novodvorsky~\cite{Bessel} showed that there is at most one subspace $\mathcal{B}(\pi_v) \subseteq \mathcal{B}$ such that the right regular representation of $G(F_v)$ on $\mathcal{B}(\pi_v)$ is equivalent to $\pi_v$.
If the subspace $\mathcal{B}(\pi_v)$ exists, then it is called the $\nu_v \otimes \psi_{T,v}$ Bessel model of $\pi_v$.

\subsection{Archimedean Local Bessel Models}

Now suppose $v$ is an infinite place of $F$. Let $K_v$ be the maximal compact subgroup of $G(F_v)$. Let $\mathcal{B}$ be the vector space of functions $\phi : G(F_v) \rightarrow \mathbb{C}$ with the following properties~\cite{pitaleschmidt2009}:
\begin{enumerate}
 \item $\phi$ is smooth and $K_v$-finite.
 \item $\phi(rg)= \nu_v \otimes \psi_{T,v}(r) \phi(g)$ for all $r \in R(F_v)$ and all $g \in G(F_v)$.
 \item $\phi$ is slowly increasing on $Z(F_v) \backslash G(F_v)$.
\end{enumerate}

Let $\pi_v$ be a $(\mathfrak{g_v}, K_v)$-module with space $V_{\pi_v}$. Suppose that there is a subspace $\mathcal{B}(\pi_v) \subset \mathcal{B}$, invariant under right translation by $\mathfrak{g}_v$ and $K_v$, and is isomorphic as a $(\mathfrak{g}_v, K_v)$-module to $\pi_v$, then $\mathcal{B}(\pi_v)$ is called the $\nu_v \otimes \psi_{T,v}$ Bessel model of $\pi_v$. In some instances the Bessel model at an archimedean place is known to be unique. For example, when $v$ is a real place and $\pi_v$ is a lowest or highest weight representation of $GSp_4(\mathbb{R})$ the Bessel model of $\pi_v$ is unique~\cite{pitaleschmidt2009}. It is also known to be unique when the central character of $\pi_v$ is trivial~\cite{Bessel}. The results of this article do not depend on the uniqueness of the Bessel model at any archimedean place; however, if the model is not unique, then there is no local integral.

\subsection{Bessel Coefficients}
Let $\mathcal{A}_0(G)$ be the space of cuspidal automorphic forms on $G(\mathbb{A})$. Suppose that $\pi$ is an irreducible cuspidal automorphic representation of $G(\mathbb{A})$ with space $V_\pi \subset \mathcal{A}_0(G)$. Let $\omega_\pi$ denote the central character of $\pi$. Let $\phi \in V_\pi$.

Suppose that $\nu$ is as above. Denote by $Z_{\mathbb{A}}$ the center of $G(\mathbb{A})$ so $Z_{\mathbb{A}} \subset H_T(\mathbb{A})$. Suppose that $\nu_{|Z_{\mathbb{A}}}= \omega_\pi^{-1}$. Define the $\nu \otimes \psi_T $ Bessel coefficient of $\phi$ to be
\begin{align}
\phi^{T,\nu}(g) = \int \limits_{Z_{\mathbb{A}} R(F) \backslash R(\mathbb{A})} (\nu\otimes \psi_T)^{-1}(r) \phi(rg) dr. \label{besseldef}
\end{align}

\section{Siegel Eisenstein Series} \label{siegeleisenstein}
For more details about Siegel Eisenstein series of $\Sp_{2n}$ see Kudla and Rallis~\cite{kudlarallis1994} and Section 1.1 of Kudla, Rallis, and Soudry~\cite{kudlarallissoudry1992}.

\begin{defn}[Induced Representation] \label{induced}
The induced representation of $\delta_P^{\frac{1}{3}(s-\frac{1}{2})}$ to $G(\mathbb{A})$ is defined to be
\begin{equation*}
\text{Ind}_{P(\mathbb{A})}^{G(\mathbb{A})}( \delta_P ^{\frac{1}{3}(s-\frac{1}{2})})=\left\{ \begin{array}{rl} f:G(\mathbb{A}) \rightarrow \mathbb{C} \Big| & f \ \text{is smooth, right $K$-finite, and for}\\  \Big| & p \in P(\mathbb{A}), \, f(pg)=\delta_P^{\frac{1}{3}(s+1)}(p)f(g) \end{array} \right\}. 
\end{equation*}
\end{defn}

For convenience write Ind$(s)=\text{Ind}_{P(\mathbb{A})}^{G(\mathbb{A})}( \delta_P ^{\frac{1}{3}(s-\frac{1}{2})})$. Ind$(s)$ is a representation of \linebreak $(\mathfrak{g}_\infty , K_\infty ) \times  G(\mathbb{A}_{\text{fin}})$ under right translation. A standard section $f(s, \cdot)$ is one such that its restriction to $K$ is independent of $s$. Let $f(s, \cdot ) \in \text{Ind}(s)$ be a holomorphic standard section. That is for all $g \in G(\mathbb{A})$ the function $s \mapsto f(s, g)$ is a holomorphic function.

For a finite place $v$ define $f_v^\circ(s, \cdot)$ to be the function so that $f_v^\circ(s,k)=1$ for $k \in K_v$.

There is an intertwining operator
\begin{equation*}
 M(s): \text{Ind}(s) \rightarrow \text{Ind}(1-s).
\end{equation*}
For $Re(s) >2$, $M(s)$ may be defined by means of the integral~\cite{kudlarallis1988}*{4.1}
\begin{equation*}
 M(s)f(s,g):= \int \limits_{N(\mathbb{A})} f(s, wng) \, dn
\end{equation*}
where
\begin{equation*}
w=\begin{bmatrix} & & 1 & \\ & & & 1\\ -1 & & & \\ & -1 & & \end{bmatrix}.
\end{equation*}
The induced representation factors as a restricted tensor product with respect to $f_v^\circ(s, \cdot)$:
\begin{equation*}
\text{Ind}(s)={\bigotimes_v}^\prime \text{Ind}_v(s),
\end{equation*}
and so does the intertwining operator
\begin{equation*}
M(s)=\bigotimes_v M_v(s).
\end{equation*}
There is a normalization of $M_v(s)$
\begin{equation*}
M^*_v(s)=\frac{\zeta_v(s+1) \, \zeta_v(2s)}{\zeta_v(s-1) \, \zeta_v(2s-1)}M_v(s)
\end{equation*}
where $\zeta_v(\cdot)$ is the local zeta factor for $F$ at $v$, so that
\begin{equation*}
M^*_v(s)f^\circ_v(s,g)= f^\circ_v(1-s,g).
\end{equation*}

Define the Siegel Eisenstein series
\begin{equation*}
 E(s ,f,g)= \sum \limits_{\gamma \in P(F)\backslash G(F)} f(s,\gamma g)
\end{equation*}
which converges uniformly for $Re(s)>2$ and has meromorphic continuation to all $\mathbb{C}$~\cite{kudlarallis1994}. Furthermore, the Eisenstein series satisfies the functional equation
\begin{equation*}
E(s,f,g)=E(1-s,M(s)f,g)
\end{equation*}
~\cite{kudlarallis1994}*{1.5}.
Later, it will be useful to work with the normalized Eisenstein series. Let $S$ be a finite set of places, including the archimedean places, such that for $v \notin S$ $f_v=f_v^{\circ}$. Define
\begin{equation}
 E^*(s,f,g)=\zeta^S(s+1) \, \zeta^S (2s) E(s,f,g). \label{normalizing}
\end{equation}
Kudla and Rallis completely determined the locations of possible poles of Siegel Eisenstein series~\cite{kudlarallis1994}. The normalized Eisenstein series $E^*(s,f,g)$ has at most simple poles at $s_0=1,2$ ~\cite{kudlarallis1994}*{Theorem 1.1}.

\section{The Weil Representation}

\subsection{The Schr\"odinger Model}\label{schrodinger}
%The main references for this part are Rallis~\cite{rallis84} and the survey article of Kudla~\cite{kudla84}.

Consider the orthogonal space $V_T$ with symmetric form $( \, , )_T$, and the four dimensional symplectic space $W$ with symplectic form $<\, ,>$. Let $\mathbb{W}=V_T \otimes W$ be the symplectic space with form $\ll \, , \gg$ defined on pure tensors by $\ll u\otimes v, u' \otimes v' \gg \, =(u,u')_T \, <v,v'>$ and extended to all of $\mathbb{W}$ by linearity.
The Weil representation $\omega=\omega_{\psi_T^{-1}}$ is a representation of $\widetilde{Sp}(\mathbb{W})$. However, restricting this representation to $\widetilde{Sp}(W) \times O(V_T) \hookrightarrow \widetilde{Sp}(\mathbb{W})$. Since the dimension of $V_T$ is even, there is a splitting $Sp(W) \times O(V_T) \hookrightarrow \widetilde{Sp}(W) \times O(V_T)$~\cite{rallis1982}*{Remark 2.1}.

Suppose that $X$ is a maximal isotropic subspace of $W$. Then $\mathbb{X} = X \otimes_F V_T$ is a maximal isotropic subspace of $\mathbb{W}$. The space of the Schr\"odinger model, $\mathcal{S} (\mathbb{X})$, is the space of Schwartz-Bruhat functions on $\mathbb{X}$. Let $v$ be a place of $F$. If $v$ is a finite place, then $\mathcal{S}(\mathbb{X}(F_v))$ is the space of locally constant functions with compact support. If $v$ is an infinite place, then $\mathcal{S}(\mathbb{X}(F_v))$ is the space of $C^\infty$ functions all derivatives of which are rapidly decreasing.

Identify $\mathbb{X}$ with $V_T^2=\text{Mat}_{2}$.

The local Weil representation at a finite place $v$ restricted to 
\begin{equation*}
Sp(W)(F_v) \times O(V_T)(F_v)
\end{equation*}
acts in the following way on the Schr\"odinger model
\begin{align*}
\omega_v(1, h) \varphi(x) &=\varphi(h^{-1}x),\\
\omega_v(m(a), 1) \varphi(x) &= \chi_{T,v}\circ det(a) \ |\mathrm{det}(a)|_v \ \varphi(xa),\\
\omega_v( n(X), 1) \varphi(x) &= \psi_{  ^tx T x}^{-1}(X) \varphi(x),\\
\omega_v( w, 1) \varphi (x) & = \gamma \cdot \hat{\varphi}(x).
\end{align*}
where $\gamma$ is a certain eighth root of unity, and $\hat{\varphi}$ is the Fourier transform of $\varphi$ defined by
\begin{align*}
\hat{\varphi}(x)= \int \limits_{V_T(F_v)^2} \varphi(x') \psi( (x, x')_1 ) dx'.
\end{align*}
Here $( \ , \ )_1$ is defined as follows: 
for $
x, y \in \mathbb{X}=\text{Mat}_{2}$ define
\begin{equation*}
 (x , y)_1:=tr (x \cdot y).
\end{equation*}
Note that matrices of the form $m(a)$, $n(X)$, and $w$ generate $Sp_4$.

The space $\mathcal{S}(\mathbb{X}(\mathbb{A}))$ is spanned by functions $\varphi= \otimes_v \varphi_v$ where $\varphi_v=\varphi_v^\circ$ is the normalized local spherical function for all but finitely many of the finite places $v$. At an unramified place $\varphi_v^\circ=1_{\mathbb{X}(\mathcal{O}_v)}$. The global Weil representation, $\omega={\otimes_v}^\prime \omega_v$, is the restricted tensor product with respect to the normalized spherical functions $\varphi_v^\circ$.

Suppose that $F_v=\mathbb{R}$. Assume that $\psi_T= \exp(2\pi i x)$. Let $K_{1,v}=\text{Sp}_4(\mathbb{R}) \cap O_4(\mathbb{R})$.  Let $V_T^+$ and $V_T^-$ be positive definite and negative definite, respectively, subspaces of $V_T(F_v)$ such that $V_T(F_v)=V_T^+ \oplus V_T^-$. For $x \in V_T$ define
\begin{equation*}
(x,x)_+=      \left\{
\begin{array}{rl}
 (x,x) & \text{if } x \in V_T^+ \\
- (x,x) & \text{if } x \in V_T^- \\
\end{array} \right.
\end{equation*}
For $x \in V_T^2$ let $(x,x)=( (x_i, x_j)_{i,j})\in V_T^2$. Define \begin{equation*}\varphi_v^\circ(x)= \exp(-\pi \, \tr((x,x)_+)).\end{equation*}

Now, suppose $F_v=\mathbb{C}$. Assume that $\psi_T=\exp(4 \pi i (x+\bar{x})$. In this case $K_{1,v}\cong \Sp(4)$, the compact real form of $\Sp(4, \mathbb{C})$. There is a choice of basis so \begin{equation*}(x,x)_+= {}^t \bar{x}x, \end{equation*} and \begin{equation*}\varphi_v^\circ(x)=\exp(-2\pi \, \tr((x,x)_+)).\end{equation*}

The subspace of $K_{1,v}$ finite vectors in the space of smooth vectors, $\mathcal{S}_0(\mathbb{X}(F_v)) \subset \mathcal{S}(\mathbb{X}(F_v))$, consists of functions of the form $p(x) \varphi_v^\circ $ where $p$ is a polynomial on $V_T(F_v)^2$.

\subsection{Extension to Similitude Groups}
Harris and Kudla describe how to extend the Weil representation to similitude groups~\cite{harriskudla1992}*{$\S$3}. See also~\cite{harriskudla2004} and~\cite{roberts2001}.

The Weil representation can be extended to the group 
\begin{align*}
Y=\{ (g, h) \in GSp_4 \times GSO(V_T) \ | \ \lambda_G(g)=\lambda_{T}(h) \}.
\end{align*}

For $(g,h) \in Y$ the action of $\omega_v$ is defined by
\begin{align*}
\omega_v(g,h) \varphi(x)= |\lambda_{T}(h)|_v^{-1} \ \omega_v( g_1 , 1) \varphi(h^{-1}x)
\end{align*}
where
\begin{align*}
g_1=\begin{bmatrix} I_2 & \\ & \lambda_G(g)^{-1} \cdot I_2 \end{bmatrix}g.
\end{align*}

Note that the natural projection to the first coordinate 
\begin{align*}
p_1: &Y \rightarrow GSp(4)\\
&(g,h) \mapsto g
\end{align*}
is generally not a surjective map. Indeed, $g \in Im(p_1)$ if and only if there is an $h \in GSO(V_T)$ such that $ \lambda_G(g)=\lambda_T(h)$. Define
\begin{equation*}
 G^+ := p_1 (Y).
\end{equation*}

\subsection{Theta Lifts} \label{thetalifts}
Let $H=GSO(V_T)$, and $H_1=SO(V_T)$.

\begin{defn}
The theta lift of $\nu^{-1}$ to $G^+(\mathbb{A})$ is given by the integral  \label{theta}
\begin{equation*}
\theta_\varphi(\nu^{-1})(g)=
\int \limits_{H_1(F) \backslash H_1(\mathbb{A})} \sum \limits_{x \in V_T^2(F)} \omega(g, h_g h_1) \varphi (x) \nu^{-1} (h_g h_1 ) dh_1.
\end{equation*}
\end{defn}
Here, $h_g \in H(\mathbb{A})$ is any element so that $\lambda_{T}(h_g)=\lambda_G(g)$. Note that Definition~\ref{theta} is independent of the choice $h_g$.
Since $H_1(F) \backslash H_1(\mathbb{A})$ is compact the integral is termwise absolutly convergent~\cite{weil1965}. 

There is a natural inclusion
\begin{equation*}
 G(F)^+ \backslash G(\mathbb{A})^+ \hookrightarrow G(F) \backslash G(\mathbb{A}).
\end{equation*}
Consider $\theta_\varphi(\nu^{-1})$ as a function of $G(F) \backslash G(\mathbb{A})$ by extending it by $0$~\cite{ganichino}*{$\S$7.2}.

If $\varphi$ is chosen to be a $K$-finite Schwartz-Bruhat function, then $\theta_\varphi(\nu^{-1})$ is a $K$-finite automorphic form on $G(F) \backslash G(\mathbb{A})$ \cite{harriskudla1992}.

\section{The Degree Five $L$-function} \label{lfunctionsec}
The connected component of the dual group of $\GSp_4$ is $^L G^\circ= \text{GSp}_4(\mathbb{C})$~\cite{borel1979}*{I.2.2 (5)}.
The degree five $L$-function of $\GSp_4$ corresponds to the map of $L$-groups~\cite{soudry1988}*{page 88}
\begin{equation*}
 \varrho: \GSp_4(\mathbb{C}) \rightarrow \PGSp_4(\mathbb{C}) \cong \SO_5(\mathbb{C}).
\end{equation*}
I describe the local $L$-factor explicitly when $v$ is finite and $\pi_v$ is equivalent to an unramified principal series.
Consider the maximal torus $A_0$ of $G$ and an element $t \in A_0$:
\begin{equation}
t=\text{diag}(a_1, a_2, a_0 a_1^{-1}, a_0 a_2^{-1}):=\begin{bmatrix} a_1 & & & \\ & a_2 & & \\ & & a_0 a_1^{-1} & \\ & & & a_0 a_2^{-2} \end{bmatrix}. \label{toruselement}
\end{equation}
The character lattice of $G$ is 
\begin{equation*}
 X=\mathbb{Z}e_0 \oplus \mathbb{Z}e_1 \oplus \mathbb{Z}e_2
\end{equation*}
where $e_i(t)=a_i$.
The cocharacter lattice is
\begin{equation*}
 X^{\vee}= \mathbb{Z}f_0 \oplus \mathbb{Z}f_1 \oplus \mathbb{Z}f_2
\end{equation*}
where 
\begin{align*}
 &f_0(u)=\text{diag} (1,1,u,u),  &  f_1(u)=\text{diag}( u, 1, u^{-1}, 1),\\ & f_2(u)=\text{diag}(1,u,1,u^{-1}).
\end{align*}

Suppose
\begin{equation*}
\pi_v \cong \pi_v(\chi)=Ind_{B(F_v)}^{G(F_v)}(\chi)
\end{equation*}
 where 
\begin{equation}
\chi(t)=\chi_1(a_1) \chi_2(a_2) \chi_0(a_0), \label{chidefine}
\end{equation}
 and $t$ is given by~\eqref{toruselement}. Then $^L G^\circ=\hat{G}$ has character lattice $X^\prime=X^\vee$ and cocharacter lattice $X^{\prime \vee}=X$. Let $f_i^\prime=e_i \in X^{\prime \vee}$.
Define 
\begin{equation}
\hat{t}=\prod_{i=0}^{3} f_i^\prime(\chi_i(\varpi_v)) \in  {^L G ^\circ}. \label{satakeparameter}
\end{equation}

Then $\hat{t}$ is the Satake parameter for $\pi_v(\chi)$~\cite{asgarischmidt2001}*{Lemma 2}.  The  Langlands $L$-factor is defined in~\cite{borel1979}*{II.7.2 (1)} to be
\begin{IEEEeqnarray*}{rCl}
L(s, \pi_v, \varrho):&=&\det( I - \varrho(\hat{t}) q_v^{-s}) ^{-1} \nonumber \\
&=&(1-q_v^{-s})^{-1}(1- \chi_1(\varpi_v) q_v^{-s})^{-1} (1- \chi_1(\varpi_v)^{-1}q_v^{-s})^{-1}\nonumber \\&&(1- \chi_2(\varpi_v) q_v^{-s})^{-1} (1- \chi_2(\varpi)^{-1} q_v^{-s})^{-1}.
\end{IEEEeqnarray*}

Let $S$ be a finite set of primes, including the archimedean primes, such that if $v \notin S$, then $\pi_v$ is \text{unramified}. Then the partial $L$-function is defined to be
\begin{equation*}
 L^S(s, \pi)=L^S(s, \pi, \varrho)=\prod_{v \notin S} L(s, \pi_v, \varrho).
\end{equation*}
The product converges absolutely for $Re(s) \gg 0$~\cite{langlands1971}.

\section{Global Integral Representation} \label{global}
 The main result of this section is Theorem~\ref{eulerproduct} which states that the integral unfolds as an Euler product of local integrals.

As before $G=$GSp$_4$, $G_1=$Sp$_4$, $P=MN$ is the Siegel parabolic subgoup of $G$, and let $P_1=M_1 N=P \cap G_1$ where $M_1 = M \cap G_1$.

The global integral is
\begin{align}
I(s; f, \phi, \nu)= I(s):&= \int \limits_{Z_\mathbb{A} G(F) \backslash G(\mathbb{A})}
E(s,f, g) \phi(g) \theta_\varphi(\nu^{-1})(g) \, dg\\
&=\int \limits_{Z_\mathbb{A} G(F)^+ \backslash G(\mathbb{A})^+}
E(s, f, g) \phi(g) \theta_\varphi(\nu^{-1})(g) \, dg \label{10}
\end{align}
where equality holds because $\theta_\varphi(\nu^{-1})$ is supported on $G(F)^+ \backslash G(\mathbb{A})^+$.
The central character of $E(s, f, - )$ is trivial, and the central character of $\theta_\varphi (\nu^{-1})=\omega_\pi ^{-1}$, so the integrand is $Z_\mathbb{A}$ invariant. Since $E(s, f, -)$ and $\theta_\varphi(\nu^{-1})$ are automorphic forms, they are of moderate growth. Since $\phi$ is a cuspidal automorphic form, it is rapidly decreasing on a Siegel domain~\cite{moeglinwaldspurger1995}*{I.2.18}. Therefore, the integral \eqref{10} converges everywhere that $E(s, f, -)$ does not have a pole.

Define 
\begin{align*}
\mathbb{A}^{\times, +}:=\lambda_T(H(\mathbb{A})),  & & F^{\times, +} := F^\times \cap \mathbb{A}^{\times, +} \subseteq \mathbb{A}^{\times,+},\\
\mathbb{A}^{\times, 2}:=\{a^2 | \, a \in \mathbb{A}^\times \}, & & \mathcal{C}:=\mathbb{A}^{\times, 2} F^{\times, +} \backslash \mathbb{A}^{\times, +}.
\end{align*}
There is an isomorphism
\begin{align}
Z_\mathbb{A} G_1(\mathbb{A}) G(F)^+ \backslash G(\mathbb{A})^+ \cong  \mathcal{C}. \label{quotient1}
\end{align}
The isomorphism is realized by considering the map from $G(\mathbb{A})^+ \longrightarrow \mathcal{C}$, $g \mapsto \lambda_G(g)$. It has kernel $Z_\mathbb{A} G_1(\mathbb{A}) G(F)^+$.
This fact is stated in~\cite{ganichino}.

Identify $Z_\mathbb{A}$ with the subgroup of scalar linear transformations in $H(\mathbb{A})$.

\begin{prop} \label{quotient2}
\begin{equation}
Z_\mathbb{A} H_1(\mathbb{A}) H(F) \backslash H(\mathbb{A}) \cong \mathcal{C}. \label{Hiso}
\end{equation}
\end{prop}
\begin{proof}
 Consider the map $H(\mathbb{A}) \rightarrow \mathcal{C}$, $h \mapsto \lambda_T(h)$. This map is onto by definition of $\mathbb{A}^{\times, +}$. I must show that the kernel is $Z_\mathbb{A} H_1(\mathbb{A}) H(F)$. Suppose $\lambda_T(h)=a^2 \mu$ where $a \in \mathbb{A}^{\times, 2}$ and $\mu \in F^{\times, +}$. By Hasse's norm theorem~\cite{hasse1967} there is an element $h_\mu \in H(F)$ such that $\lambda_T(h)=\mu$. Let $z(a)$ be the scalar matrix with eigenvalue $a$. Since $\lambda_T( z(a)^{-1} h h_\mu^{-1})=1$, $h_1= z(a)^{-1} h h_\mu^{-1} \in H_1(\mathbb{A})$. Therefore, $h= z(a) h_1 h_\mu$. This shows that $Z_\mathbb{A} H_1(\mathbb{A}) H(F)$ contains the kernel of this map. The opposite inclusion is obvious.  This proves the proposition.
\end{proof}

Fix sections 
\begin{align*}
&\mathcal{C} \rightarrow G(\mathbb{A})^+ & & \mathcal{C} \rightarrow H(\mathbb{A}) \nonumber \\
&c \mapsto g_c & & c \mapsto h_c
\end{align*}

\begin{prop}
 There is a measure $dc$ on $\mathcal{C}$ and measures $dh_1$ and $dg_1$ on $H_1(F) \backslash H_1(\mathbb{A})$ and $G_1(F) \backslash G_1(\mathbb{A})$, respectively, such that
\begin{equation*}
  \int \limits_{Z_\mathbb{A} H(F) \backslash H(\mathbb{A}) } \, f(h) \, dh= \int \limits_{\mathcal{C} } \int \limits_{H_1(F) \backslash H_1(\mathbb{A})} \, f(h_1 h_c) \, dh_1 \, dc,
\end{equation*}
and
\begin{equation*}
  \int \limits_{Z_\mathbb{A} G(F)^+ \backslash G(\mathbb{A})^+ } \, f(g) \, dg= \int \limits_{\mathcal{C}} \int \limits_{G_1(F) \backslash G_1(\mathbb{A})} \, f(g_1 g_c) \, dg_1 \, dc.
\end{equation*}
\end{prop}
\begin{proof}

Let $dh$ denote the right invariant measure on $Z_\mathbb{A} H(F) \backslash H(\mathbb{A})$. Then by \cite{prasadtakloobighash2011}*{Lemma 13.2} there are measures $dh_1$ and $dh_c$ so that for all  $f \in L^{1}(Z_\mathbb{A} H(F) \backslash H(\mathbb{A}))$
\begin{equation}
 \int \limits_{Z_\mathbb{A} H(F) \backslash H(\mathbb{A}) }f(h) \, dh= \int \limits_{Z_\mathbb{A} H_1(\mathbb{A}) H(F) \backslash H(\mathbb{A}) } \int \limits_{H_1(F) \backslash H_1(\mathbb{A})} f(h_1 h_c) \, dh_1 \, dh_c. \label{hintegral}
\end{equation}
Through the isomorphism ~\eqref{Hiso} define a measure $dc := dh_c$ on $\mathcal{C}$. By ~\eqref{quotient1} define a measure $dg_c:=dc$ on $Z_\mathbb{A} G_1(\mathbb{A}) G(F)^+ \backslash G(\mathbb{A})^+$. Then there is a choice of measures $dg$ and $dg_1$ so that for $f \in L^{1}(Z_\mathbb{A} G(F)^+ \backslash G(\mathbb{A})^+)$
\begin{equation}
 \int \limits_{Z_\mathbb{A} G(F)^+ \backslash G(\mathbb{A})^+ }f(g) \, dg= \int \limits_{Z_\mathbb{A} G_1(\mathbb{A}) G(F)^+ \backslash G(\mathbb{A})^+ } \int \limits_{G_1(F) \backslash G_1(\mathbb{A})} f(g_1 g_c) \, dg_1 \, dg_c. \label{gintegral}
\end{equation}
\end{proof}

Then \eqref{10} equals
\begin{equation}
\int \limits_{\mathcal{C}} \int \limits_{G_1(F) \backslash G_1(\mathbb{A})}
E(s, f, g_1 g_c) \phi(g_1 g_c) \theta_\varphi(\nu^{-1})(g_1 g_c) \, dg_1 \,  dc. \label{9}
\end{equation}

Denote the theta kernel by
\begin{equation*}
 \theta_\varphi(g_1g_c,h_1h_c)= \sum \limits_{x \in V_T^2(F)}  \omega(g_1 g_c, h_1 h_c ) \varphi(x).
\end{equation*}
Then
\begin{equation*}
\theta_\varphi(\nu^{-1})(g_1 g_c) = \int \limits_{H_1(F) \backslash H_1(\mathbb{A})}\theta_\varphi(g_1 g_c, h_1 h_c) \nu^{-1}(h_1 h_c) dh_1.
\end{equation*}

As noted in section~\ref{thetalifts} this integral converges absolutely. The following adjoint identity holds for the global theta integral
\begin{IEEEeqnarray*}{rCl}
\int \limits_{G_1(F) \backslash G_1(\mathbb{A})}
\int \limits_{H_1(F) \backslash H_1(\mathbb{A})}
E(s, f, g_1 g_c) \, \phi(g_1 g_c) \, \theta_\varphi(g_1 g_c, h_1 h_c) \, \nu^{-1}(h_1 h_c) \, dh_1\, dg_1 \nonumber \\
=
\int \limits_{H_1(F) \backslash H_1(\mathbb{A})} \int \limits_{G_1(F) \backslash G_1(\mathbb{A})}
E(s, f, g_1 g_c) \, \phi(g_1 g_c) \,  \theta_\varphi(g_1 g_c, h_1 h_c) \, \nu^{-1}(h_1 h_c) \,  dg_1 \, dh_1. \label{adjoint} \IEEEeqnarraynumspace
\end{IEEEeqnarray*}

Since $ P_1(F) \backslash G_1(F) \cong P(F) \backslash G(F)$, then 
\begin{equation*}
E(s, f,g) = \sum \limits_{\gamma \in P(F) \backslash G(F) } f(s, \gamma g) = \sum \limits_{\gamma \in P_1(F) \backslash G_1(F)} f(s, \gamma g).
\end{equation*}
Then the inner integral of~\eqref{adjoint} becomes
\begin{equation*}
 \int \limits_{P_1(F) \backslash G_1(\mathbb{A})}
f(s, g_1 g_c) \, \phi(g_1 g_c) \,  \theta_\varphi(g_1 g_c, h_1 h_c) \, \nu^{-1}(h_1 h_c)  \, dg_1.
\end{equation*}

Expanding the theta kernel gives
\begin{align}
 \int \limits_{G_1(F) \backslash G_1(\mathbb{A})}
f(s, g_1 g_c) \, \phi(g_1 g_c) \, \sum \limits_{x \in V_T^2(F)} \omega(g_1 g_c, h_1 h_c) \varphi (x) \nu^{-1} (h_1 h_c) \, dg_1 \label{1}
\end{align}
The Levi factor of $P_1$ is $M_1 \cong \GL_2$. The Weil representation restricted to this subgroup acts on $\mathcal{S}(\mathbb{A})$ by 
\begin{equation*}
 \omega(m(y), 1) \, \varphi_1(x)= |\det(y)|_\mathbb{A} \varphi_1(xy)
\end{equation*}
for $\varphi_1 \in \mathcal{S}(\mathbb{A})$, $y \in \GL_2(\mathbb{A})$, and $x \in Mat_2(\mathbb{A})$. Consider $x \in Mat_2(F)$. If $\det(x)=0$, then $Stab_{\GL_2(\mathbb{A})}(x)$ contains a normal unipotent subgroup. By the cuspidality of $\phi$ this term vanishes upon integration.

Therefore
\begin{IEEEeqnarray}{rCl}
& & \int \limits_{ P_1(F) \backslash G_1(\mathbb{A})} f(s, g_1 g_c) \, \phi(g_1 g_c)
 \sum \limits_{x \in Mat_2(F)} \omega(m(x) g_1 g_c, h_1 h_c) \varphi (1_2) \nu^{-1} (h_1 h_c ) dg_1\nonumber \\
&= & \int \limits_{ P_1(F) \backslash G_1(\mathbb{A})} f(s, g_1 g_c) \, \phi(g_1 g_c)
 \sum \limits_{x \in \GL_2(F)} \omega(m(x) g_1 g_c, h_1 h_c) \varphi (1_2) \nu^{-1} (h_1 h_c ) dg_1\nonumber \\
&= & \int \limits_{ N(F) \backslash G_1(\mathbb{A})} f(s, g_1 g_c) \, \phi(g_1 g_c)
  \omega(g_1 g_c, h_1 h_c) \varphi (1_2) \, \nu^{-1} (h_1 h_c ) \,  dg_1.  \label{a1}
\end{IEEEeqnarray}
Note that the integral
\begin{equation*}
 \int \limits_{ N(F) \backslash G_1(\mathbb{A})} f(s, g_1 g_c) \, \phi(g_1 g_c) \, \omega(g_1 g_c, h_1 h_c) \varphi (1_2) \, \nu^{-1} (h_1 h_c ) \, dg_1
\end{equation*}
is $H_1(F)$ invariant; however, the integrand is not.

Then
\begin{align*}
& \int \limits_{ N(F) \backslash G_1(\mathbb{A})} f(s, g_1 g_c) \, \phi(g_1 g_c) \omega(g_1 g_c, h_1 h_c) \varphi (1_2) \, dg_1\\
=& \int \limits_{ N(\mathbb{A}) \backslash G_1(\mathbb{A})} \int \limits_{N(F) \backslash N(\mathbb{A})} f(s,g_1 g_c)  \phi(ng_1 g_c)
\omega(n g_1 g_c, h_1 h_c) \, \varphi (1_2) \, dn \, dg_1.
\end{align*}
Define
\begin{equation*}
 \phi^T(g) :=\int \limits_{N(F) \backslash N(\mathbb{A}) } \phi(ng) \, \psi_T^{-1}(n) \, dn. 
\end{equation*}
Then
 \begin{equation*}
  \int \limits_{N(F) \backslash N(\mathbb{A})} \phi(ng)\omega(n g, h_g h) \, \varphi (1_2) \, dn \\
= \phi^{T}(g) \, \omega(g, h_g h) \varphi (1_2).
 \end{equation*}

This follows since for $n \in N(\mathbb{A})$
\begin{equation*}
\omega(n g_1 g_c, h_1 h_c) \varphi (1_2)= \psi_T^{-1}(n) \, \omega(g_1 g_c, h_1 h_c) \varphi (1_2),
\end{equation*}
so the integral \eqref{a1} becomes
\begin{align}
I(s) = & \int \limits_\mathcal{C} \int \limits_{H_1(F) \backslash H_1(\mathbb{A})} \int \limits_{ N(\mathbb{A}) \backslash G_1(\mathbb{A})} f(s,g_1 g_c) \phi^{T}(g_1 g_c) \nonumber\\
& \times \omega(g_1 g_c, h_1 h_c) \varphi (1_2) \, \nu^{-1} (h_1 h_c ) \, dh_1 dg_1 dg_c, \label{2}
\end{align}

Computing in the Weil representation
\begin{align}
\omega(g_1 g_c , h_1 h_c) \varphi(1_2) =& |\lambda_G(g_c)|^{-1}_\mathbb{A} \, \omega \left(  \ell \left(\lambda_G(g)^{-1}\right) g_1 g_c, 1 \right) \varphi \left( (h_1 h_c)^{-1}\right)\\
=& \chi_V \circ \det(h_1 h_c) \, |\lambda_G(g_c)|^{-1}_\mathbb{A} \, |\det (h_1 h_c)^{-1}|^{-1}_\mathbb{A} \nonumber \\ &\times \omega \left( m(h_1 h_c )^{-1}  \ell\left(\lambda_G(g_c)\right)g_1 g_c, 1 \right) \varphi(1_2). \label{41}
\end{align}
For $h \in H_T$, $\det(h) \in N_{E/F}(E^\times)$. Therefore, $\chi_V \circ \det(h)=1$.
 
Combining this with the fact that 
\begin{equation*}
 |\lambda_G(g_c)|_\mathbb{A} =|\det\left(h_1 h_c\right)|_\mathbb{A}
\end{equation*}
(see Section~\ref{orthog}) and applying it to $\eqref{41}$ gives
\begin{align}
 \omega(g_1 g_c , h_1 h_c) \varphi(1_2)&= \omega \left( m(h_1 h_c)^{-1} \ell \left( \lambda_G(g_c)^{-1} \right) g_1 g_c, 1 \right) \varphi(1_2) \nonumber\\
&= \omega \left( b(h_1 h_c)^{-1} g_1 g_c, 1 \right) \varphi(1_2). \label{3}
\end{align}

Since $\lambda_G( b(h_1 h_c))=\lambda_G(g_1 g_c)$, the map $g_1 \mapsto b(h_1 h_c) g_1 g_c^{-1}$ sends $G_1$ to itself.
\begin{prop}
 Let $d\bar{g}$ be the right invariant measure on $N(\mathbb{A}) \backslash G_1(\mathbb{A})$, and $g \in P(\mathbb{A}) \subseteq G(\mathbb{A})$. Then $d (\overline{ghg^{-1}})=|\delta_P(g)^{-1}|_\mathbb{A} \cdot d \bar{g}$. 
\end{prop}
\begin{proof}
Suppose $dn$ is Haar measure on $N(\mathbb{A})$ and $d\bar{g}$ is the right invariant measure on $N(\mathbb{A}) \backslash G_1(\mathbb{A})$ normalized so that for $f \in L^1( G_1(\mathbb{A}))$
\begin{equation*}
 \int \limits_{G_1(\mathbb{A})} f(g) \, dg =  \int \limits_{N(\mathbb{A}) \backslash G_1(\mathbb{A})} \int \limits_{N(\mathbb{A})} f(n \bar{g}) \, dn \, d\bar{g}  
\end{equation*}

Let $g \in P(\mathbb{A})$. The transformation $h \mapsto ghg^{-1}$ preserves Haar measure on $G_1(\mathbb{A})$. Let $E$ be a measurable subset of $G_1(\mathbb{A})$ such that with finite volume with respect to $dg_1$, and let $vol(E)$ denote this volume. Then the volume of $N(\mathbb{A}) \backslash N(\mathbb{A}) E$ is given by the formula
\begin{equation*}
 vol(N(\mathbb{A}) \backslash N(\mathbb{A}) E) = \dfrac{vol(E)}{\int \limits_{N(\mathbb{A}) \cap E} dn}.
\end{equation*}
Since $d(gng^{-1})=|\delta_P(g)|_\mathbb{A} \cdot dn$, then $d(\overline{ghg^{-1}})=|\delta_P(g)|_\mathbb{A}^{-1} \cdot d\bar{g}$.
\end{proof}
Since $\delta_P(b(h_1 h_c))=1$, the map $g_1 \mapsto b(h_1 h_c) g_1 g_c^{-1}$ preserves the right invariant measure on $N(\mathbb{A}) \backslash G_1(\mathbb{A})$. 
Substituting $\eqref{3}$ into $\eqref{2}$, and making the above change of variables gives
 
\begin{IEEEeqnarray*}{rCl}
I(s) &= & \int \limits_\mathcal{C} \int \limits_{H_1(F) \backslash H_1(\mathbb{A})} \int \limits_{ N(\mathbb{A}) \backslash G_1(\mathbb{A})} f(s,b(h_1 h_c) g_1) \phi^{T}(b(h_1 h_c) g_1) \nonumber\\
& & \qquad \times   \omega(g_1, 1 ) \varphi (1_2) \nu^{-1} (h_1 h_c ) \, dg_1 dh_1  dg_c.
\end{IEEEeqnarray*}
The $Z_{\mathbb{A}} H_1(\mathbb{A}) H(F) \backslash H(\mathbb{A})$ integral and the $H_1(F) \backslash H_1(\mathbb{A})$ fold together ($H$ is abe\-lian so $h_c h_1 = h_1 h_c$) to produce 
\begin{align}
 \int \limits_{Z_{\mathbb{A}} H(F) \backslash H(\mathbb{A})} \int \limits_{ N(\mathbb{A}) \backslash G_1(\mathbb{A})} f(s,b(h) g_1) \, \phi^{T}(b(h) g_1) \, \omega(g_1, 1 ) \varphi (1_2) \, \nu^{-1} (h_1 h_c ) \, dg_1 dh. \label{42}
\end{align}
Since $b(h) \in P(\mathbb{A})$, and $\delta_P\left( b(h) \right)=1$, $f(s,b(h) g_1)=f(s, g_1)$.
Therefore, changing the order of integration in $\eqref{42}$ and applying Proposition~\ref{quotient2} produces
\begin{align}
&\int \limits_{ N(\mathbb{A}) \backslash G_1(\mathbb{A})} f(s, g_1) \, \omega(g_1, 1 ) \varphi (1_2) \int \limits_{Z_{\mathbb{A}} H_T(F) \backslash H_T(\mathbb{A})} \phi^{T}(h g_1) \, \nu^{-1} (h) \, dh \, dg_1 \label{4} \\
= &\int \limits_{N(\mathbb{A}) \backslash G_1(\mathbb{A})} f(s,g_1) \phi^{T, \nu}(g_1) \, \omega(g_1, 1) \varphi(1_2) \, dg_1. \label{blah3}
\end{align}

The next section shows that \eqref{4} converges absolutely for $Re(s) >2$, justifying the change in order of integration.

\begin{prop} \label{eulerproduct}
Let $\phi^{T,\nu}=\otimes_v \phi^{T,\nu}_v$, $f(s, \cdot)=\otimes_v f_v(s, \cdot)$, and $\varphi=\otimes_v \varphi_v$. Then for $Re(s)>2$
\begin{align}
& \int \limits_{Z_\mathbb{A} G(F) \backslash G(\mathbb{A})}
E(s, f,g) \phi(g) \theta(\nu^{-1}, \varphi)(g) \, dg\\
=& \int \limits_{N(\mathbb{A}) \backslash  G_1(\mathbb{A})} f(s,g) \, \phi^{T, \nu}(g) \, \omega(g, 1) \varphi(1_2) \, dg\\
=& \int \limits_{N(\mathbb{A}_\infty) \backslash G_1(\mathbb{A}_\infty)} f(s,g_\infty) \, \phi^{T, \nu}(g_\infty) \, \omega(g_\infty, 1) \varphi(1_2) \, dg_\infty \cdot \prod \limits_{v < \infty} I_v(s) \label{blah1111}
\end{align}
where
\begin{align}
I_v(s)=\int \limits_{N(F_v) \backslash G_1(F_v)} f_v(s, g_v) \, \phi_v ^{T, \nu}(g_v) \, \omega_v(g_v, 1) \varphi_v(1_2) \, dg_v. \label{blah2222}
\end{align}
\end{prop}
The uniqueness of the Bessel model is used to obtain the factorization in \eqref{blah1111}. When the local archimedean Bessel models are unique, the integral factors at these places as in \eqref{blah2222}.

\section{Absolute Convergence of the Unfolded Integral} \label{abs}

\begin{prop}
The integral
\begin{equation*}
  \int \limits_{N(\mathbb{A}) \backslash G_1(\mathbb{A})}  \int \limits_{Z_{\mathbb{A}} H_T(F) \backslash H_T(\mathbb{A}) } f(s,g) \, \phi^{T}(h g) \, \nu^{-1}(h) \, \omega(g, 1) \varphi(1_2) \, dh \, dg \label{absconv}
\end{equation*}
converges absolutely for $Re(s)>2$.
\end{prop}

\begin{proof}
This argument follows~\cite{moriyama2004} to show that $\phi$ is bounded on $G_1(\mathbb{A})$. By~\cite{moeglinwaldspurger1995}*{Corollary I.2.12, I.2.18}, $\phi$ is rapidly decreasing. To be precise, suppose $\mathfrak{S}$ is a Siegel domain for $G(\mathbb{A})$. Let $G(\mathbb{A})^1 := \cap_{\chi} \ker |\chi|_\mathbb{A}$ where $\chi$ range over rational characters of $G$. Then
\begin{equation*}
 G(\mathbb{A})^1=\{ g \in G(\mathbb{A}) \Big| |\lambda_G(g)|_\mathbb{A}=1 \}.
\end{equation*}
Therefore, $G_1(\mathbb{A}) \subset G(\mathbb{A})^1$.

\begin{defn}[A rapidly decreasing function on $G(\mathbb{A})$~\cite{moeglinwaldspurger1995}*{I.2.12}]\label{rapiddecrease}
A function $\phi : \mathfrak{S} \rightarrow \mathbb{C}$ is rapidly decreasing if there exists an $r>0$ such that for all real positive valued characters $\lambda$ of the standard maximal torus $A_0$, there exists $C_0>0$ such that for all $z \in Z_\mathbb{A}$ and $g \in G(\mathbb{A})^1 \cap \mathfrak{S}$ the follwing inequality holds
\begin{equation}
 |\phi(zg)| \leq C_0 ||z||^r \lambda( a(g)) \label{siegel}
\end{equation}
where $|| \cdot ||$ is the height function on $G(\mathbb{A})$, and $a(g)$ is defined so that if $g=nak$, then $a(g)=a$ where $n \in N_0$, the unipotent radical of the Borel, $a \in A_0$, and $k \in K$.
\end{defn}
 By choosing $z=1$, and $\lambda$ to be the adelic norm of the similitude character in~\eqref{siegel}, then the right hand side of the inequality equals $C_0$ for $g \in \mathfrak{S} \cap G_1(\mathbb{A})$. Therefore, $\phi$ is bounded on $\mathfrak{S} \cap G_1(\mathbb{A})$. However, $\phi$ is $G(F)$ invariant, so $\phi$ is bounded on \begin{equation*}G(F)(\mathfrak{S} \cap G_1(\mathbb{A}))\supseteq G_1(\mathbb{A}).\end{equation*}

The quotients $Z_{\mathbb{A}} H_T(F) \backslash H_T(\mathbb{A})$ and $N(F) \backslash N(\mathbb{A})$ are compact. Therefore, \begin{equation}|\nu(r)|=|\psi_T(n)|=1 \end{equation} for all $r \in H_T(\mathbb{A}) \cap G_1(\mathbb{A})$, and all $n \in N(\mathbb{A})$.
Assume that all representatives $r \in Z_{\mathbb{A}} H_T(F) \backslash H_T(\mathbb{A})$ are chosen so that $r \in G_1(\mathbb{A})$. Then $r g_1 \in G_1(\mathbb{A})$ and 
\begin{equation}
 |\phi(r g_1)| \, |\nu^{-1}(r)| < C_0. \label{ineq}
\end{equation}
Furthermore, since $\nu_{|Z_\mathbb{A}}$ agrees with the central character of $\phi$, \eqref{ineq} holds for all $r \in H_T(\mathbb{A})$.
Then
\begin{align*}
\int \limits_{Z_{\mathbb{A}} H_T(F) \backslash H_T(\mathbb{A}) } |\phi^T( h g)| \, |\nu^{-1}(h)| \, dh &
=\int \limits_{Z_\mathbb{A} R(F) \backslash R(\mathbb{A})} | (\nu \otimes \psi_T)^{-1}(r)| \, | \phi(r g_1) | \, dr \nonumber \\
 &\leq \text{vol}\left(Z_\mathbb{A} R(F) \backslash R(\mathbb{A})\right) \cdot C_0.
\end{align*}

Therefore,
\begin{align*}
  \int \limits_{N(\mathbb{A}) \backslash G_1(\mathbb{A})}  \int \limits_{Z_{\mathbb{A}} H_T(F) \backslash H_T(\mathbb{A}) } |f(s,g)| \, |\phi^{T}(h g)| \, |\nu^{-1}(h)| \, |\omega(g, 1) \varphi(1_2)| \, dh \, dg \nonumber \\
\leq C \int \limits_{N(\mathbb{A}) \backslash G_1(\mathbb{A})} |f(s,g)| \, |\omega(g, 1) \varphi(1_2)| \, dg.
\end{align*}
The Schwartz-Bruhat function $\varphi$ is $K$-finite, as is $f(s, -)$, so there is some open subgroup $K_0 \leq K$ such that $[K : K_0]=n < \infty$, and $\varphi$ and $f(s, -)$ are $K_0$-invariant. Let $\{ k_i \}_{1 \leq i \leq n}$ be a set of irredundant coset representatives for $K / K_0$. We have
\begin{equation*}
 G_1(\mathbb{A})=P_1(\mathbb{A}) K
\end{equation*}
Suppose that $p=m(a) n \in P_1(\mathbb{A})$, and $k \in k_i K_0$. Define $\varphi_i:=\omega(k_i,1)\varphi$. Then we have
\begin{align*}
\omega(pk, 1)\varphi(1_2)&=\omega(p, 1) \omega(k, 1) \varphi(1_2)\\
&=\omega(p) \varphi_i(1_2)\\
&=\psi_T(n) \, \chi_V \circ \det(a) \, |\det(a)|_{\mathbb{A}} \, \varphi_i(a).
\end{align*}
Therefore,
\begin{align*}
  &\int \limits_{N(\mathbb{A}) \backslash G_1(\mathbb{A})} |f(s,g)| \, |\phi^{T, \nu}(g)| \, |\omega(g, 1) \varphi(1_2)| \, dg\\ \leq &
\int \limits_{N(\mathbb{A}) \backslash P_1(\mathbb{A})} \int \limits_{K} |\delta_P(p)^{-1}| \, |\delta_P(p)^{s/3+1/3}||f(s,k)| \, |\omega(pk, 1)\varphi(1_2)| \, dp \, dk\\
\leq & \text{vol}(K_0) \times \, \sum \limits_{i=1}^{n} |f(s,k_i)| \int \limits_{GL_2(\mathbb{A})}  |\varphi_i(a)| \, |\det(a)|^{s-1}_{\mathbb{A}} \, da.
\end{align*}
Absolute convergence of $\eqref{absconv}$ depends only on the convergence of
\begin{equation*}
 \int \limits_{GL_2(\mathbb{A})} |\varphi_i(a)| \, |\det(a)|^{s-1}_{\mathbb{A}} \, da.
\end{equation*}
The Schwartz-Bruhat function $\varphi_i=\otimes \varphi_{i,v}$ is rapidly decreasing, i.e. $\varphi_{i,v}$ is compactly supported at each finite places $v$, and $\varphi_i$ are rapidly decreasing when $v$ is archimedean. Let $Q$ be the Borel subgroup of $GL_2$, and let $L=\prod \limits_v L_v$, where $L_v$ is the maximal compact subgroup of $GL_2(F_v)$ so that GL$_2=QL$. There is a compact finite index open subgroup $L_i \leq L$ such that $\varphi_i$ is $L_i$ invariant. Let $\varphi_{ij}$, $j=1, \ldots, m$, be the $L$ translates of $\varphi_i$.  Then
\begin{align*}
  & \quad \int \limits_{GL_2(\mathbb{A})} |\varphi_i(a)| |\det(a)|^{s-1}_{\mathbb{A}} \, da\\ 
&=  \text{vol}(L_i) \times \sum \limits_{j=1}^{m} \, \int \limits_{Q(\mathbb{A})} |\varphi_{ij}(b)| \, |\det(b)|^{s-1}_{\mathbb{A}} \, db.
\end{align*}

Then
\begin{align}
 &\int \limits_{Q(\mathbb{A})} |\varphi_{ij}(b)| |\det(b)|^{s-1}_{\mathbb{A}} db\\
=&\int \limits_{\mathbb{A}^{\times}} \int \limits_{\mathbb{A}^\times} \int \limits_{\mathbb{A}} \, \left|\varphi_i \begin{pmatrix} a_1 & x\\ & a_2 \end{pmatrix} \right| \, |a_1|^{s-1}_{\mathbb{A}} |a_2|_{\mathbb{A}}^{s-1} \, \left|\frac{a_1}{a_2}\right|^{-1}_{\mathbb{A}} \, dx \, \frac{da_1}{|a_1|_{\mathbb{A}}} \, \frac{da_2}{|a_2|_{\mathbb{A}}} \\
=& \int \limits_{\mathbb{A}^{\times}} \int \limits_{\mathbb{A}^\times} \int \limits_{\mathbb{A}} \, \left|\varphi_i \begin{pmatrix} a_1 & x\\ & a_2 \end{pmatrix} \right| \, |a_1|^{s-3}_{\mathbb{A}} |a_2|^{s-1}_{\mathbb{A}} \, dx \, da_1 \, da_2. \label{estimate}
\end{align}
Since $\phi_i$ decreases rapidly as $|a_1|_{\mathbb{A}}$, $|a_2|_{\mathbb{A}}$, and $|x|_{\mathbb{A}}$ become large, the integral $\eqref{estimate}$ converges for $\text{Re}(s) > 2$.
\end{proof}

\begin{cor} \label{boundedlemma}
There exists a real number $C$ so that for every $g_1 \in G_1(\mathbb{A})$, 
\begin{equation}
|\phi^{T, \nu}(g_1)| \leq C.
\end{equation}
\end{cor}

\begin{remark} Absolute convergence of the integral right of the line Re$(s) =2$ is the best one could hope for since the Eisenstein series $E(s, f, -)$ has a possible pole at $s=2$ by Section~\ref{siegeleisenstein}, and~\cite{kudlarallis1994}*{Theorem 1.1}.
\end{remark}

\section{Computation of the Unramified Integral}
\label{unramifiedchapter}
The local integral is
\begin{equation}
I_v(s)= \int \limits_{N(F_v) \backslash G_1(F_v)} f_v(s, g) \, \phi_v ^{T, \nu}(g) \, \omega_v(g, 1) \varphi_v(1_2) \, dg. \label{11}
\end{equation}
Let $v$ be a finite place. As before $T= \begin{bmatrix} 1 & \\ & -\rho \end{bmatrix}$.
\begin{defn}\label{unramifieddef}
The data for the integral $I_v(s)$ are unramified if all of the following hold:
\begin{enumerate}[1.]
\item $K_v = G(\mathcal{O}_v)$, and by the $\mathfrak{p}$-adic Iwasawa decomposition $G(F_v)=P(F_v) K_v$
\item $\phi_v^{T, \nu}=\phi_v^{T, \nu^\circ}$ is the normalized local spherical Bessel function, i.e. it is right $K_v$ invariant
\item $\varphi_v =\varphi_v^\circ = \mathbf{1}_{Mat_{2, 2}(\mathcal{O}_v)}$ is the normalized spherical function for the Weil representation
\item $f_v(g, s)=f_v^\circ(g,s)=\delta_{P,v} ^{\frac{s}{3}+\frac{1}{3}}(g)$ where the modulus character is extended to the entire group $G(F_v)$ by $\delta_{P,v}(pk)=\delta_{P,v}(p)$ for $p \in P(F_v)$ and $k \in K_v$
\item $\nu( H_T(\mathcal{O}_v))=1$
\item $\rho \in \mathcal{O}_v^\times$
\end{enumerate}
\end{defn}
Assume that all the data are unramified for $I_v(s)$. This is the case for almost every $v$.

Let $P_1=P \cap G_1$, and $M_1=M \cap G_1 \cong GL_2$, and $K_{1,v}=K_v \cap G_1(F_v)$.
With these assumptions the integrand of $\eqref{11}$ is constant on double cosets \newline
$N(F_v) \backslash G_1(F_v) / K_{1,v}$.  By the $\mathfrak{p}$-adic Iwasawa decomposition, $G_1=P_1(F_v) K_{1,v}$, and since \begin{equation*}M_1(F_v) \cong N(F_v) \backslash P_1(F_v)\end{equation*} representatives may be found among representatives for $M_1(F_v) / \left( M_1(F_v) \cap K_{1,v} \right).$ 
By~\cite{furusawa1993}
\begin{equation}
 \GL_2(F)= \coprod \limits_{m \geq 0} H(F_v) \begin{bmatrix} \varpi^m & \\ & 1 \end{bmatrix} \GL_2(F_v). \label{decompF}
\end{equation}

Let $m \geq 0$, and define 
\begin{equation*}
 H^m(\mathcal{O}_v) :=H(F_v) \cap \begin{bmatrix} \varpi^{m} & \\ & 1 \end{bmatrix} \GL_2(F_v) \begin{bmatrix} \varpi^{-m} & \\ & 1 \end{bmatrix} .
\end{equation*}
Note that  $H^m(\mathcal{O}_v) \subseteq H(\mathcal{O}_v)$.

Recall that $E$ is the discriminant field of $V_T$. Define $E_v := E \otimes_F F_v$. Let  $\left( \frac{E}{v} \right)$ denote the Legendre symbol which equals $-1$, $0$, or $1$ according to whether $v$ is inert, ramifies, or splits in $E$. 
By elementary number theory, $\left( \frac{E}{v} \right) =0$ for only finitely many primes $v$. This case is not considered. Call the case when $\left( \frac{E}{v} \right) =-1$ as the \textit{inert case}, and the case when $\left( \frac{E}{v} \right) =1$ as the \textit{split case}. In each case $E_v^\times \cong H(F_v)$.
If $\left( \frac{E}{v} \right)=-1$, then $E_v / F_v$ is an unramified quadratic extension.
If $\left( \frac{E}{v}\right)=+1$, then $E_v \cong F_v \oplus F_v$. In this case there is an isomorphism
\begin{equation*}
 \iota : H(F_v) \rightarrow (F_v \oplus F_v)^\times.
\end{equation*}
Let $\Pi_1 := \iota^{-1} ( (\varpi, 1))$, and $\Pi_2 := \iota^{-1}( (1, \varpi))$. Then $\det \Pi_i \in \mathfrak{p}$ for $i=1,2$, and $\Pi_1 \Pi_2 = diag(\varpi, \varpi)$.

\subsection{The Inert Case}

\begin{prop}
 If $\left( \frac{E}{v} \right) = -1$, then a complete set of irredundant coset representatives for $N(F_v)\backslash G_1(F_v) / K_1$ is given by 
\begin{equation*}
 m\left( h \begin{bmatrix} \varpi^{m+n} & \\ & \varpi^n \end{bmatrix} \right)
\end{equation*}
where $m\geq 0$, $n \in \mathbb{Z}$, and $h$ runs over a set of representatives for $H(\mathcal{O}_v) / H^m(\mathcal{O}_v)$.
\end{prop}
\begin{proof}
 This follows from the above decomposition \eqref{decompF} and the fact that $H(F_v)=Z(F_v)H(\mathcal{O}_v)$ where $Z$ is the center of $\GL_2$.
\end{proof}

Since 
\begin{equation*}
\omega_v(m(g),1)\varphi_v^\circ (1_2)=\chi_{T,v}\circ \det(g) \, |\det(g)|_v \, \varphi_v^\circ(g),
\end{equation*}
only $m(g)$ with $g \in \text{GL}_2(F_v) \cap \text{Mat}_{2, 2} (\mathcal{O}_v)$ are in the support of the \begin{equation*}\omega_v(m(g) , 1) \varphi_v^\circ (1_2).\end{equation*} Hence, a complete set of irredundant representatives for the cosets in the support of $\omega_v(m(g) , 1) \varphi_v(1_2)$ is given by the representatives listed above with $n \geq 0$.
\begin{prop}
The various components of the integrand are computed as follows
\begin{align}
&\delta_P ^{\frac{s}{3}+\frac{1}{3}} \left( m \left( h  \begin{bmatrix} \varpi_v^{n+m} & \\ & \varpi^n \end{bmatrix} \right) \right)=q_v^{-(2n+m)(s+1)} \label{integrand1}\\
&\omega_v \left( m \left(h \begin{bmatrix} \varpi_v^{n+m} &  \\ & \varpi_v^n\\ \end{bmatrix} \right), 1 \right) \varphi_v^\circ(1_2)=\chi_{T,v}(\varpi_v^{2n+m}) q_v^{-(2n+m)} \label{integrand2}\\
&\text{vol}\left(N(F_v) \backslash N(F_v) \,  m\left( h \begin{bmatrix} \varpi^{n+m} & \\ & \varpi^n \end{bmatrix} \right) \, K_{1,v}\right)=q_v^{6n+3m} \label{integrand3}
\end{align}
\end{prop}

\begin{proof}
The proof of \eqref{integrand1} is just an application of \eqref{adeq}, and \eqref{integrand2} follows from Section~\ref{schrodinger}.

To prove \eqref{integrand3} observe that there are measures $dn$ on $N$ and $d\bar{g}$ on $N \backslash G_1$ so that
\begin{equation*}
 \int \limits_{G_1} f(g) dg = \int \limits_{N \backslash G_1} \int \limits_{N} f(n \bar{g}) \, dn \, d\bar{g}
\end{equation*}
and
\begin{equation*}
 \int \limits_{G_1} f(nmk) dg = \int \limits_{M_1} \int \limits_{ K} \int  \limits_{N} \delta_{P}^{-1}(m) \, f(nmk) \, dn\, dk \, dm.
\end{equation*}
Therefore, $\delta_{P,v}^{-1} \cdot d\dot{g}$ gives a right invariant measure on $N(F_v) \backslash G_1(F_v)$ normalized so that vol $\left( N(F_v) \backslash N(F_v)K_{1,v} \right) =1$.

Again, let $h$ be a representative for an element of $H(\mathcal{O}_v) / H^m(\mathcal{O}_v)$, and let
\begin{equation*}
A(h,m,n)=N(F_v) \backslash N(F_v) \,  m\left( h  \begin{bmatrix} \varpi_v^{m+n} &  \\ & \varpi_v^{n}\\ \end{bmatrix} \right) \, K_{1,v}.
\end{equation*}
Then
$\text{vol}(A(h,m,n))=\delta_{P,v}^{-1}\left( m\left( h  \begin{bmatrix} \varpi_v^{m+n} &  \\ & \varpi_v^n\\ \end{bmatrix} \right) \right)=q_v^{3m+6n}$
\end{proof}

Note that the integrand does not depend on the coset of $H(\mathcal{O}_v) / H^{m}(\mathcal{O}_v)$.
\begin{lem}[Furusawa]
 The index $ [ H(\mathcal{O}_v) : H^{m}(\mathcal{O}_v) ]=q^m(1-\left( \frac{E}{v}\right)\frac{1}{q})$ for $m \geq 1$.
\end{lem}

Then the local integral $\eqref{11}$ is
\begin{IEEEeqnarray}{rCl}
 I_v(s) &=& (1+\frac{1}{q})  \sum \limits_{m,n \geq 0} \phi^{T, \nu} _v \left( m \left( \begin{bmatrix} \varpi^{n+m} & \\ & \varpi^n \end{bmatrix} \right) \right) \chi_T(\varpi_v^{m}) q_v^{2n(1-s)}q_v^{m(2-s)} \nonumber \\
&& -\ \frac{1}{q} \sum \limits_{n \geq 0} \phi^{T, \nu} _v \left( m \left( \begin{bmatrix} \varpi^{n} & \\ & \varpi^n \end{bmatrix} \right) \right) q_v^{2n(1-s)}. \label{16}
\end{IEEEeqnarray}

\subsection{The Split Case}

\begin{prop}
If $\left( \frac{E}{v} \right) = +1$, then a complete set of irredundant coset representatives for $N(F_v)\backslash G_1(F_v) / K_1$ are given by 
\begin{equation*}
m\left( h \, \Pi_i^k \begin{bmatrix} \varpi^{m+n} & \\ & \varpi^{n} \end{bmatrix} \right)
\end{equation*}
where $i=1,2$, $m,k\geq 0$, $n \in \mathbb{Z}$, and $h$ are representatives for $H(\mathcal{O}_v) / H^m(\mathcal{O}_v)$.
\end{prop}
\begin{proof}
For every $(x,y) \in (F_v \oplus F_v)^\times$ there are unique integers $k_1$, and $k_2$ such that $(x,y)=(\varpi^{k_1}, \varpi^{k_2}) \cdot u$ where $u \in \mathcal{O}_v^\times \oplus \mathcal{O}_v^\times$. If $k_1 > k_2$, then let $i=1$. Otherwise $i=2$. Let $n=\min\{k_1, k_2\}$, and $k=k_i-n$. The result then follows from the decomposition \eqref{decompF}.
\end{proof}

\begin{prop} The various components of the integrand are computed as follows
\begin{align*}
&\delta_P ^{\frac{s}{3}+\frac{1}{3}} \left( m \left( h \, \Pi_i^k \begin{bmatrix} \varpi_v^{n+m} & \\ & \varpi^n \end{bmatrix} \right) \right)=q_v^{-(2n+m+k)(s+1)}\\
&\omega_v \left( m \left(h \, \Pi_i^k \begin{bmatrix} \varpi_v^{n+m} &  \\ & \varpi_v^n\\ \end{bmatrix} \right), 1 \right) \varphi_v^\circ(1_2)=\chi_{T,v}(\varpi_v^{2n+m+k}) q_v^{-(2n+m)} \\
&\text{vol}\left(N(F_v) \backslash N(F_v) \,  m\left( h \, \Pi^k_i\begin{bmatrix} \varpi^{n+m} & \\ & \varpi^n \end{bmatrix} \right) \, K_{1,v}\right)=q_v^{6n+3m+3k}
\end{align*}
\end{prop}
\begin{proof}
The only nontrivial part is volume computation which follows from an argument that is similar to the proof of ~\cite{furusawa1993}*{Lemma 3.5.3}. 
\end{proof}

Then the local integral $\eqref{11}$ is
\begin{align}
 I_v(s) = (1-\frac{1}{q}) & \sum \limits_{i=1,2} \sum \limits_{m,n \geq 0} \phi^{T, \nu} _v \left( m \left( \Pi_i^k \begin{bmatrix} \varpi^{n+m} & \\ & \varpi^n \end{bmatrix} \right) \right)  q_v^{(2n+k)(1-s)}q_v^{m(2-s)} \nonumber \\
+ \frac{1}{q} & \sum \limits_{i=1,2} \sum \limits_{n \geq 0} \phi^{T, \nu} _v \left( m \left(\Pi_i^k \begin{bmatrix} \varpi^{n} & \\ & \varpi^n \end{bmatrix} \right) \right) q_v^{(2n+k)(1-s)}. \label{16a}
\end{align}

The expressions \eqref{16} and \eqref{16a} are evaluated using Sugano's Formula.

\section{Sugano's Formula}
The results of this section were obtained by Sugano~\cite{sugano1985}, but I follow the treatment found in Furusawa~\cite{furusawa1993}.

Define 
\begin{align*}
 h_v(\ell,m)=\begin{bmatrix} \varpi_v^{2m+\ell} & & & \\ & \varpi_v^{m+\ell} & & \\ & & 1 & \\ & & & \varpi_v^{m} \end{bmatrix}.
\end{align*}
The local spherical Bessel function is supported on double cosets
\begin{equation*}
\coprod \limits_{\ell,m \geq 0} R(F_v) h_v(\ell,m) GSp_4(\mathcal{O}_v).
\end{equation*}
In~\cite{sugano1985} Sugano explicitly computes the following expression when $\phi^{T,\nu}$ is spherical:
\begin{align*}
C_v(x,y)=\sum \limits_{\ell,m \geq 0} \phi_v^{T,\nu}(h_v(\ell, m)) x^m y^\ell.
\end{align*}
Since $\pi_v$ is assumed to be a spherical representation, it is isomorphic to an unramified principal series representation. I describe this more precisely.

Let $P_0$ be the standard Borel subgroup of $G$ with Levi component $M_0$.
\begin{align*}
M_0=\left\{ \begin{bmatrix} a_1 & & & \\ & a_2 & & \\ & & a_3 & \\ & & & a_4 \end{bmatrix} \Bigg| a_1 a_3 = a_2 a_4 \right\}.
\end{align*}
There exists a character
\begin{equation*}
\gamma_v : M_0(F_v) \rightarrow \mathbb{C}^\times
\end{equation*}
that is trivial on $M_0(\mathcal{O}_v)$ such that $\pi_v \cong \mathrm{Ind}^{G(F_v)}_{P_0(F_v)}(\gamma_v)$. Then $\gamma_v$ is determined by its values
\begin{align*}
\gamma_{1,v}=\gamma_v \begin{bmatrix} \varpi_v & & & \\ & \varpi_v & & \\ & & 1 &\\ & & & 1 \end{bmatrix}, & \quad
\gamma_{2,v}=\gamma_v \begin{bmatrix} \varpi_v & & & \\ & 1 & & \\ & & 1 &\\ & & & \varpi_v \end{bmatrix},\\
\gamma_{3,v}=\gamma_v \begin{bmatrix} 1 & & & \\ & 1 & & \\ & & \varpi_v &\\ & & & \varpi_v \end{bmatrix}, & \quad
\gamma_{4,v}=\gamma_v \begin{bmatrix} 1 & & & \\ & \varpi_v & & \\ & & \varpi_v &\\ & & & 1 \end{bmatrix}.\\
\end{align*}
Note that
\begin{equation*}
\gamma_{1,v} \gamma_{3,v} = \gamma_{2,v} \gamma_{4,v}=\omega_{\pi,v}(\varpi_v).
\end{equation*}

Let
\begin{equation*}
\epsilon_v=      \left\{
\begin{array}{ll}
 0 & \text{if } \left( \frac{E}{v} \right)=-1,\\
\nu(\varpi_{E,v}) & \text{if } \left( \frac{E}{v} \right)=0,\\
 \nu(\varpi_{E,v})+\nu(\varpi_{E,v}^\sigma) \qquad & \text{if } \left( \frac{E}{v} \right)=1.
\end{array} \right.
\end{equation*}
\begin{theorem}(Sugano)
\begin{equation*}
C_v(x,y)=\frac{H_v(x,y)}{P_v(x)Q_v(y)},
\end{equation*}
where
\begin{IEEEeqnarray*}{rCl}
P_v(x)&=& (1-\gamma_{1,v} \gamma_{2,v} q_v^{-2}x)(1-\gamma_{1,v} \gamma_{4,v} q_v^{-2}x) \nonumber\\ && (1-\gamma_{2,v} \gamma_{3,v} q_v^{-2}x)(1-\gamma_{3,v} \gamma_{4,v} q_v^{-2}x),\\
Q_v(y) &=& \prod \limits_{i=1}^4 (1-\gamma_{i,v} q_v^{-3/2}y),\\
H_v(x,y) &=& (1+A_2 A_3 x y^2)\{M_1(x)(1+A_2 x)+A_2 A_5 A_1^{-1} \alpha x^2 \}\nonumber \\ &&-A_2 x y \{\alpha M_1(x) -A_5 M_2(x) \} -A_5 P_v(x)y -A_2 A_4 P_v(x) y^2,\\
M_1(x) &=& 1 -A_1^{-1}(A_1+A_4)^{-1}(A_1 A_5 \alpha + A_4 \beta -A_1 A_5^2 -2 A_1 A_2 A_4)x \nonumber \\ &&+A_1^{-1} A_2^2 A_4 x^2,\\
M_2(x) &=& 1+A_1^{-1}(A_1 A_2 -\beta)x + A_1^{-1}A_2 (A_1 A_2 -\beta)x^2+A_2^3x^3,
\end{IEEEeqnarray*}
\begin{align*}
\alpha=&q_v^{-3/2}\sum \limits_{i=1}^4 \gamma_{i,v},&
 \quad \beta=&q_v^{-3}\sum \limits_{1 \leq i < j \leq 4} \gamma_{i,v} \gamma_{j,v},\\
A_1=&q_v^{-1},&  \quad A_2=&q_v^{-2} \nu(\varpi_v),\\
\quad A_3=&q_v^{-3} \nu(\varpi_v),& \quad A_4=&-q_v^{-2}\left( \frac{E}{v} \right),\\
A_5=&q_v^{-2} \epsilon_v.
\end{align*}

\end{theorem}

The parameters $\gamma_{i,v}$ differ from the parameters of Section \ref{lfunctionsec}. One verifies that
\begin{align*}
\gamma_{1,v}=\chi_1 \chi_2 \chi_0(\varpi_v), & &
\gamma_{2,v}=\chi_1 \chi_0(\varpi_v),\\
\gamma_{3,v}=\chi_0(\varpi_v), & &
\gamma_{4,v}=\chi_2 \chi_0(\varpi_v),
\end{align*}
and
\begin{equation*}
 \omega_{\pi,v} = \chi_1 \chi_2 \chi_0^2.
\end{equation*}
Therefore,
\begin{align}
 \gamma_{1,v} \gamma_{2,v} = \chi_1^2 \chi_2 \chi_0^2 (\varpi_v)= \chi_1 \omega_{\pi,v}(\varpi_v), \nonumber\\
 \gamma_{1,v} \gamma_{4,v} = \chi_1 \chi_2^2 \chi_0^2(\varpi_v) = \chi_1 \omega_{\pi,v}(\varpi_v), \nonumber\\
 \gamma_{2,v} \gamma_{3,v} = \chi_1  \chi_0^2(\varpi_v) = \chi_2^{-1} \omega_{\pi,v}(\varpi_v), \nonumber\\
 \gamma_{3,v} \gamma_{4,v} = \chi_2  \chi_0^2(\varpi_v) = \chi_1^{-1} \omega_{\pi,v}(\varpi_v).\label{blah124}
\end{align}

\begin{prop}\label{absconv1}
Let $P_1$, $P_2 \in X \cdot \mathbb{C}[X]$, i.e. $P_i$ have constant coefficient $0$, then $C_v( P_1(q^{-s}) ,P_2( q^{-s}))$ converges absolutely for Re$(s) \gg 0$. Therefore, the terms of this series may be rearranged without affecting the sum.
\end{prop}

The proposition is a consequence of the following lemma which is the local version of Corollary \ref{boundedlemma}.
\begin{lem}
 For each place $v$ of $F$, there is a constant $A_v>0$ and a real number $\alpha$ independent of $v$ so that
\begin{equation*}
 | \phi_v^{T,\nu}(g_v)| \leq A_v | \lambda_G(g_v)|_v ^\alpha.
\end{equation*}
\end{lem}
\begin{proof}
Again, I follow \cite{moriyama2004}.
Pick a place $w | \infty$. Then for $g \in G(\mathbb{A})$ one may write $g=z_w g_1$ where $g_1 \in G(\mathbb{A})^1$, and $z_w$ is in the center of $G(F_w)$.
Then by Corollary \ref{boundedlemma}, I estimate
\begin{align*}
 | \phi^{T,\nu} (g)|=& | \omega_{\pi,w}(z_w)|_w |\phi^{T,\nu}(g_1)|_\mathbb{A} \nonumber\\
=& A \cdot |z_w|_w ^\beta \nonumber \\
=& A \cdot |\lambda_G(g)|_\mathbb{A}^{\beta /2}.
\end{align*}
Let $g_0 \in G(\mathbb{A})$ so that $\pi^{T,\nu} \neq 0$. Then for each place $v$ define
\begin{equation*}
 A_v := A \times \prod \limits_{v^\prime \neq v} \frac{|\lambda_G(g_{0,v^\prime})|_{v^\prime} ^{\beta/2}}{|\phi^{T,\nu}(g_{0,v^\prime})|_\mathbb{A}},
\end{equation*}
and let $\alpha= \beta/2$.
\end{proof}

If $Re(s)$ is sufficiently large (to account for $A_v$ and the coefficients of $P_i$), then comparing the series $C_v(P_1(q^{-s}), P_2(q^{-s}))$ to a doubly geometric series completes the proof of the proposition.

\section{Computing the Local Integral}
\subsection{The Inert Case}
It is necessary to express $\eqref{16}$ as a linear combination of terms of the form $C_v(x,y)$ where $x$ and $y$ are monomials in $q_v^{-s}$. Since
\begin{align*}
m \begin{bmatrix} \varpi_v^{m+n} & \\ & \varpi_v^n \end{bmatrix} &= \begin{bmatrix} \varpi_v^{m+n} & & & \\ & \varpi_v^{n} & & \\ & & \varpi_v^{-n-m} & \\ & & & \varpi_v^{-n} \end{bmatrix} \nonumber \\
&=z(\varpi_v^{-n-m}) h_v(2n, m),
\end{align*}
then \eqref{16} gives
\begin{align*}
 I_v(s) = (1+\frac{1}{q}) & \sum \limits_{m,n \geq 0} \omega_{\pi,v}(\varpi)^{-m-n} \, \phi^{T, \nu} _v (h_v(m, 2n)) \chi_T(\varpi_v)^m q_v^{2n(1-s)}q_v^{m(2-s)} \nonumber \\
- \frac{1}{q} &\sum \limits_{n \geq 0}  \omega_{\pi,v}(\varpi)^{-n}  \, \phi^{T, \nu} _v (h_v(0, 2 n)) q_v^{2n(1-s)} 
\end{align*}
\begin{prop}
When $\left( \frac{E}{v}\right)=-1$ the unramified local integral is 
\begin{align*}
I_v(s)=& (1+\frac{1}{q_v})\sum \limits_{\ell, m \geq 0} (-\omega_{\pi,v}(\varpi_v)^{-1} q_v^{2-s})^m(\omega_{\pi,v}(\varpi_v)^{-1/2}q_v^{1-s})^{2 \ell} \phi^{T,\nu}_v(h_v(2\ell, m))\nonumber \\
&-\frac{1}{q_v}\sum \limits_{\ell \geq 0} (-\omega_{\pi, v}(\varpi_v)^{-1/2}q_v^{1-s})^{2 \ell} \phi^{T,\nu}_v (h_v(2 \ell, 0))\nonumber\\
=&(1+\frac{1}{q_v}) \Big[ \frac{1}{2} C_v\left( -\omega_{\pi,v}(\varpi_v)^{-1}q_v^{2-s}, \omega_{\pi, v}(\varpi_v)^{-1/2}q_v^{1-s} \right)\nonumber \\& \qquad +
\frac{1}{2} C_v\left( -\omega_{\pi,v}(\varpi_v)^{-1}q_v^{2-s}, -\omega_{\pi, v}(\varpi_v)^{-1/2}q_v^{1-s} \right) \Big] \nonumber \\& -
\frac{1}{q_v} \Big[ \frac{1}{2}  C_v\left( 0, \omega_{\pi, v}(\varpi_v)^{-1/2}q_v^{1-s} \right)+ \frac{1}{2} C_v\left( 0, -\omega_{\pi, v}(\varpi_v)^{-1/2}q_v^{1-s} \right) \Big].
\end{align*}
\end{prop}
This expression was evaluated using Sugano's formula and Mathematica.
\begin{prop}
When $\left( \frac{E}{v}\right)=-1$ the unramified local integral is
\begin{IEEEeqnarray}{rCl} 
I_v(s)&=&(1-q_v^{-s})(1-q_v^{-s-1}) \nonumber \\
&& (1+\gamma_1 \gamma_2 \omega_{\pi,v}(\varpi_v)^{-1} q_v^{-s})^{-1}(1+\gamma_1 \gamma_4 \omega_{\pi,v}(\varpi_v)^{-1} q_v^{-s})^{-1}\nonumber\\
&& (1+\gamma_2 \gamma_3 \omega_{\pi,v}(\varpi_v)^{-1} q_v^{-s})^{-1}(1+\gamma_3 \gamma_4 \omega_{\pi,v}(\varpi_v)^{-1} q_v^{-s})^{-1}. \label{blah125}
\end{IEEEeqnarray}
\end{prop}
Correcting by the normalizing factor~\eqref{normalizing} gives
\begin{prop} \label{inertprop1}
When $\left( \frac{E}{v}\right)=-1$, the normalized unramified local integral is

\begin{align*}
\zeta_v(s+1) \zeta_v(2s) I_v(s)=L(s, \pi_v \otimes \chi_{T,v})
\end{align*}
\begin{proof}
 This follows from comparing \eqref{blah124} and \eqref{blah125}.
\end{proof}

\end{prop}

\subsection{The Split Case}

Since $\chi_{T,v}(\varpi_v)=\left( \frac{E}{v} \right)=1$, $\chi_{T,v}$ does not appear in this part of the calculation. Since
\begin{align*}
m \left(\Pi_i^k \, \begin{bmatrix}\varpi_v^{m+n} & \\ & \varpi_v^n \end{bmatrix}\right) 
&=b(\Pi_i^k)
 \begin{bmatrix} \varpi_v^{m+n} & & & \\ & \varpi_v^{n} & & \\ & & \varpi_v^{-n-m-k} & \\ & & & \varpi_v^{-n-k} \end{bmatrix} \nonumber \\
&=b(\Pi_i^k) z(\varpi^{-m-n-k}) \begin{bmatrix}\varpi_v^{2m+2n+k} & & & \\ & \varpi_v^{m+2n+k} & & \\ & & 1 & \\ & & & \varpi_v^{m} \end{bmatrix} \nonumber\\
&=b(\Pi_i^k) z(\varpi_v^{-m-n-k}) h_v(2n+k, m), 
\end{align*}
so \eqref{16a} becomes
\begin{IEEEeqnarray}{rCl}
 I_v(s) &=&  (1-\frac{1}{q}) \sum \limits_{i=1,2} \sum \limits_{m,n,k \geq 0} \omega_{\pi,v}(\varpi)^{-m-n-k} \, \nu(\Pi_i)^k \, \phi^{T, \nu} _v (h_v(m, 2n+k)) \, q_v^{(2n+k)(1-s)}q_v^{m(2-s)} \nonumber \\
& & + \frac{1}{q} \sum \limits_{i=1,2} \sum \limits_{n,k \geq 0} \omega_{\pi,v}(\varpi)^{-n-k} \, \nu(\Pi_i)^k \, \phi^{T, \nu} _v (h_v(0, 2 n+k)) \, q_v^{(2n+k)(1-s)}. \label{A}
\end{IEEEeqnarray}
First, suppose $\nu(\Pi_1) \neq \nu(\Pi_2)$ which is equivalent to $\nu(\Pi_i)^2 \neq \omega_{\pi,v}(\varpi)$.
\begin{prop} \label{splitprop}
 \begin{equation}
  \frac{\omega_{\pi,v}(\varpi)^{-1} \nu(\Pi_1)}{\omega_{\pi,v}(\varpi)^{-1} \nu(\Pi_1)^2-1}+\frac{\omega_{\pi,v}(\varpi)^{-1} \nu(\Pi_2)}{\omega_{\pi,v}(\varpi)^{-1} \nu(\Pi_2)^2-1}=0 \label{identity1}
 \end{equation}
and
\begin{equation}
\frac{\omega_{\pi,v}(\varpi)^{-1}\nu(\Pi_1)^2}{\omega_{\pi,v}(\varpi)^{-1} \nu(\Pi_1)^2-1} +
\frac{\omega_{\pi,v}(\varpi)^{-1}\nu(\Pi_2)^2}{\omega_{\pi,v}(\varpi)^{-1} \nu(\Pi_2)^2-1}=1. \label{identity2}
\end{equation}
\end{prop}
\begin{proof}
 Both identities follow from the fact that $\nu(\Pi_1) \cdot \nu(\Pi_2) = \omega_{\pi,v}(\varpi_v)$.
\end{proof}
Let
\begin{align*}
 \eta_1:= \frac{\omega_{\pi,v}(\varpi)^{-1}\nu(\Pi_1)^2}{\omega_{\pi,v}(\varpi)^{-1} \nu(\Pi_1)^2-1}& & \eta_2:= \frac{\omega_{\pi,v}(\varpi)^{-1}\nu(\Pi_2)^2}{\omega_{\pi,v}(\varpi)^{-1} \nu(\Pi_2)^2-1}.\\
 \theta_1:=\frac{\omega_{\pi,v}(\varpi)^{-1}\nu(\Pi_1)}{\omega_{\pi,v}(\varpi)^{-1} \nu(\Pi_1)^2-1}& & \theta_2:= \frac{\omega_{\pi,v}(\varpi)^{-1}\nu(\Pi_2)}{\omega_{\pi,v}(\varpi)^{-1} \nu(\Pi_2)^2-1}.
\end{align*}
Then combining terms with $2n+k=\ell$ gives
\begin{IEEEeqnarray}{rCl}
 I_v(s)&=&\sum \limits_{i=1,2} (1-\frac{1}{q})  \eta_i \sum \limits_{\ell,m \geq 0} (\omega_{\pi,v}(\varpi)^{-1} \nu(\Pi_i) q^{1-s} )^\ell (\omega_{\pi,v} (\varpi)^{-1}q^{2-s})^m \phi_v^{T,\nu}(h_v(2 \ell, m)) \nonumber \\
 && -\ (1-\frac{1}{q}) (\eta_i -1) \sum \limits_{\ell, m \geq 0} (\omega_{\pi,v}(\varpi)^{-1/2} q^{1-s})^{2 \ell} ( \omega_{\pi,v}(\varpi)^{-1} q^{2-s})^m \phi_v^{T,\nu}(h_v(2\ell, m))\nonumber \\
&& -\ (1-\frac{1}{q}) \theta_i \sum \limits_{\ell,m \geq 0} (\omega_{\pi,v}(\varpi)^{-1/2} q^{1-s})^{2 \ell+1} ( \omega_{\pi,v}(\varpi)^{-1} q^{2-s})^m \phi_v^{T,\nu}(h_v(2\ell+1, m))  \nonumber \\
&& +\ \frac{1}{q} \eta_i \sum \limits_{\ell \geq 0} (\omega_{\pi,v}(\varpi)^{-1} \nu(\Pi_i) q^{1-s} )^\ell  \phi_v^{T,\nu}(h_v( \ell, 0))  \nonumber\\
&& -\ \frac{1}{q}  (\eta_i-1) \sum \limits_{\ell \geq 0} (\omega_{\pi,v}(\varpi)^{-1/2} q^{1-s})^{2 \ell}  \phi_v^{T,\nu}(h_v(2\ell, 0)) 
\nonumber \\
&& -\ \frac{1}{q}  \theta_i \sum \limits_{\ell \geq 0} (\omega_{\pi,v}(\varpi)^{-1/2} q^{1-s})^{2 \ell+1} \phi_v^{T,\nu}(h_v(2\ell+1, 0)).
 \label{splitsum}
\end{IEEEeqnarray}
Applying Proposition~\ref{splitprop} to \eqref{splitsum} gives
\begin{prop} When $\left(\frac{E}{v}\right)=+1$ and $\nu(\Pi_1) \neq \nu(\Pi_2)$ the unramified local integral is given by
\begin{align}
I_v(s)=&\left(1-\frac{1}{q_v}\right)\eta_1 \cdot C_v(\omega_{\pi,v}(\varpi_v)^{-1} q_v^{2-s}, \omega_{\pi,v}(\varpi_v)^{-1}\nu(\Pi_1)q_v^{1-s}) \nonumber \\
+& \left(1-\frac{1}{q_v}\right) \eta_2 \cdot C_v(\omega_{\pi,v}(\varpi_v)^{-1} q_v^{2-s}, \omega_{\pi,v}(\varpi_v)^{-1}\nu(\Pi_2)q_v^{1-s}) \nonumber\\
-&\frac{1}{q_v} \eta_1 \cdot C_v(0, \omega_{\pi,v}(\varpi_v)^{-1}\nu(\Pi_1)q_v^{1-s})\nonumber \\
-&\frac{1}{q_v} \eta_2 \cdot C_v(0, \omega_{\pi,v}(\varpi_v)^{-1}\nu(\Pi_2)q_v^{1-s}).\label{splitsum1}
\end{align}
\end{prop}
\begin{prop}
When $\left( \frac{E}{v}\right)=+1$  the unramified local integral is
\begin{IEEEeqnarray}{rCl}
I_v(s)&=&(1+q_v^{-s})(1-q_v^{-s-1}) \nonumber \\
&&(1-\gamma_1 \gamma_4 \omega_{\pi,v}(\varpi_v)^{-1} q_v^{-s})^{-1}(1-\gamma_1 \gamma_2 \omega_{\pi,v}(\varpi_v)^{-1} q_v^{-s})^{-1} \nonumber \\
&&(1-\gamma_2 \gamma_3 \omega_{\pi,v}(\varpi_v)^{-1} q_v^{-s})^{-1}
(1-\gamma_3 \gamma_4 \omega_{\pi,v}(\varpi_v)^{-1} q_v^{-s})^{-1}.\label{blah127}
\end{IEEEeqnarray}
\end{prop}
\begin{proof}
 When $\nu(\Pi_1) \neq \nu(\Pi_2)$ this follows from applying Sugano's formula to \eqref{splitsum1}, and the computation was verified with Mathematica. To extend the identity to all values of $\nu(\Pi_1)$ and $\nu(\Pi_2)$ observe that the right hand side of \eqref{A} is equal to
\begin{align*}
 \sum \limits_{i=1,2} \sum \limits_{m,n,k \geq 0} A(m) \phi_{v}^{T,\nu} ( h_v(m, n+2k)) X^m  Y^n Z^k,
\end{align*}
with $X=\omega_{\pi,v}^{-1}(\varpi) q_v^{2-s}$, $Y=\omega_{\pi,v}^{-1}(\varpi) q_v^{2-2s}$, $Z=\omega_{\pi,v}^{-1}(\varpi) \nu(\Pi_i) q_v^{1-s} $, $A(0)=1$, and $A(m)=1-\frac{1}{q}$ otherwise. This is an absolutely convergent power series for $Re(s) \gg 0$, and convergence is uniform as $\nu(\Pi_1)$ and $\nu(\Pi_2)$ vary in a compact set. Furthermore, for $\nu(\Pi_1) \neq \nu(\Pi_2)$ the sum is equal to a rational function in $q^{-s}$ that remains continuous for all values of $\nu(\Pi_1)$ and $\nu(\Pi_2)$ such that $\omega_{\pi,v}(\varpi) \nu(\Pi_1) \nu(\Pi_2)=1$. Then by uniform convergence the equality holds for all such values of $\nu(\Pi_1)$ and $\nu(\Pi_2)$.
\end{proof}
\begin{prop}\label{splitprop1}
When $\left( \frac{K}{v}\right)=+1$, the normalized unramified local integral is
\begin{align*}
\zeta_v(s+1) \zeta_v(2s) I_v(s)=L(s, \pi_v \otimes \chi_{T,v}).
\end{align*}
\end{prop}
\begin{proof}
This follows from comparing \eqref{blah124} and \eqref{blah127}
\end{proof}
\section{Ramified Integrals at Finite Places} \label{ramified}
Consider the local integral at a finite place $v$ where some of the data may be ramified.
\begin{equation*}
 I_v(s)= \int \limits_{N(F_v) \backslash G_1(F_v)} f_v(s, g) \, \phi_v ^{T, \nu}(g) \, \omega_v(g, 1) \varphi_v(1_2) \, dg.
\end{equation*}
Let $K_{P,v}=P(F_v) \cap K_v$. Let $K_{P_1,v}=P_1(F_v) \cap K_v= K_1 \cap K_{P,v}$.
\begin{prop}\label{induced2}
Let $h : K_{P,v} \backslash K_v \rightarrow \mathbb{C}$ be a locally constant (i.e. smooth) function. Then there exists $f_v(s,k) \in \text{Ind}(s)$ such that for all $k \in K_v$, $f(s,k)=h(k)$.
\end{prop}
\begin{proof}
This proposition follows from \cite{casselman1995}*{Proposition 3.1.1}.
\end{proof}
\begin{prop}\label{finiteram}
 There exists a $K_v$-finite section $f_v(s,g) \in \text{Ind}(s)$, and a $K_v$-finite Schwartz-Bruhat function $\varphi_v \in \mathcal{S}(\mathbb{X}(F_v))$ such that $I_v(s) \equiv 1$.
\end{prop}
\begin{proof}
This argument comes directly from~\cite{piatetski-shapirorallis1988}.

Suppose that $K_0$ is an open compact subgroup of $G_1(F_v)$ so that $\phi^{T,\nu}_v$ is right $K_0$ invariant. Consider the isomorphism $p: M_1 \rightarrow \text{GL}_2$, $m(a) \mapsto a$.
Let $K_\phi$ be an open compact subgroup of GL$_2(F_v)$ such that $K_\phi \subseteq p(M_1(F_v) \cap \nolinebreak K_0) \cap \ker \chi_T \circ \det$.
Let $\varphi_v = 1_{K_\phi}$. Since $\varphi_v$ is a smooth function, there exists an open compact subgroup $K' \subseteq G_1(F_v)$ such that $\omega_v(k,1) \varphi_v = \varphi_v$. 
Let $\mathcal{K}=K_{P_1,v} \backslash K_{P_1,v} \cdot (K' \cap K_0)$. By Proposition~\ref{induced2} there is $f_v(s, -) \in \text{Ind}(s)$ so that $f_v(s, -)|_{K_1}=1_{K_{{P_1},v} \cdot (K^\prime \cap K_0)}.$ Then
\begin{IEEEeqnarray}{rCl}
 I_v(s)&=& \int \limits_{N(F_v) \backslash G_1(F_v)} f_v(s, g) \, \phi_v ^{T, \nu}(g) \, \omega_v(g, 1) \varphi_v(1_2) \, dg \nonumber \\
&=& \int \limits_{K_{P_1,v} \backslash K_{1,v}} \int \limits_{\text{GL}_2(F_v)} \delta_P(m(a))^{-1}  f_v(s, m(a)k) \phi_v^{T,\nu} (m(a)k) \nonumber\\ && \qquad \omega_v(m(a)k, 1) \varphi_v(1_2) \, da \, dk \nonumber \\
&=& \int \limits_{K_{P_1,v} \backslash K_{1,v}} f_v(s,k) \int \limits_{\text{GL}_2(F_v)} |\det(a)|_v^{s-2} \phi_v^{T,\nu} (m(a)k)\, \omega_v(m(a)k, 1) \varphi_v(1_2) \, da \, dk \nonumber \\
&=& \int \limits_{\mathcal{K}} \int \limits_{\text{GL}_2(F_v)} |\det(a)|_v^{s-2} \phi_v^{T,\nu} (m(a)k) \, \omega_v(m(a)k, 1) \varphi_v(1_2) \, da \, dk \nonumber \\
&=& \int \limits_{\mathcal{K}} \int \limits_{\text{GL}_2(F_v)} |\det(a)|_v^{s-1} \phi_v^{T,\nu}(m(a) k) \, \chi_T( \det(a)) \, 1_{K_\phi}(a) \, da \, dk \label{blah2}
%&=& C_1 \cdot \int \limits_{\text{GL}_2(F_v)} |\det(a)|_v^{s-1} \phi_v^{T,\nu}(m(a)) \, \chi_T( \det(a)) \, \varphi_v(a) \, da. 
\end{IEEEeqnarray}
For $a \in K_\phi$, $| \det( a)|_v=1$, and $\phi_v^{T, \nu}(m(a))=1$ , then
\begin{eqnarray*}
& &\int \limits_{\mathcal{K}} \int \limits_{\text{GL}_2(F_v)} |\det(a)|_v^{s-1} \phi_v^{T,\nu}(m(a) k) \, \chi_T( \det(a)) \, 1_{K_\phi}(a) \, da \, dk \\
&=&\int \limits_{K_\phi (K' \cap K_0)} \phi_v^{T,\nu}(k) \, dk.
\end{eqnarray*}
After normalizing measures and $\phi^{T, \nu}_v$, $I_v(s) \equiv 1$.
\end{proof}

\section{Ramified Integrals at Infinite Places}\label{archimedean}
Consider the integral at the infinite places from Proposition~\ref{eulerproduct}.
\begin{equation*}
I_\infty(s)= \int \limits_{N(\mathbb{A}_\infty) \backslash G_1(\mathbb{A}_\infty)} f(s,g) \phi^{T, \nu}(g) \omega(g, 1) \varphi(1_2) \, dg.
\end{equation*}
\begin{prop} \label{archprop}
For every complex number $s_0$ there is a choice of of data $f( s, g) \in$ \text{Ind}$(s)$, and $\varphi=\varphi_\infty \otimes \varphi_{\text{fin}} \in \mathcal{S}( \mathbb{X}(\mathbb{A}))$ such that $I_\infty$ converges to a holomorphic function at $s_0$, and $I_\infty(s_0) \neq 0$.
\end{prop}
The proof of this proposition is essentially given in~\cite{piatetski-shapirorallis1988}*{\S 2}, but is reproduced here with necessary changes.

\begin{proof}
By the Iwasawa decomposition, $G_1(\mathbb{A}_\infty) = P_1(\mathbb{A}_\infty) K_\infty$, where $P_1$ has Levi factor $M_1 \cong \text{GL}_2$. The integral $I_\infty(s)$ may be broken up as a $M_1(\mathbb{A}_\infty)$ integral and a $K_\infty$ integral:
\begin{align*}
I_\infty(s)=& \int \limits_{K_\infty} \int \limits_{\text{GL}_2(\mathbb{A}_\infty)} \delta_P(m(a))^{-1} f(s, m(a)k) \phi^{T, \nu}(m(a)k) \omega(m(a)k, 1) \varphi(1_2) \, da \, dk \\
=&  \int \limits_{K_\infty} f(k,s) \int \limits_{\text{GL}_2(\mathbb{A}_\infty)} |\det(a)|_\infty^{s-2} \phi^{T, \nu}(m(a)k) \omega(m(a)k, 1) \varphi(1_2) \, da \, dk.
\end{align*}
Here $| \, |_\infty$ denotes the valuation on $ \mathbb{A}_\infty$ defined by $| x |_\infty=\prod \limits_v|x_v|_v$, and $\det(a) \in \mathbb{A}_\infty$ with $v$ coordinate equal to $\det(a_v)$.
Since $f(k ,s)$ is a standard section, it is independent of $s$ when restricted to $K_\infty$. Write $f(k,s)=f(k)$ for $k \in K_\infty$.
The integral
\begin{align}
 A(k,s):= \int \limits_{\text{GL}_2(\mathbb{A}_\infty)} |\det(a)|^{s-2} \phi^{T, \nu}(m(a)k) \, \omega(m(a)k, 1) \varphi_\infty(1_2) \, da \label{blah23}
\end{align}
gives a function on $(M_1(\mathbb{A}_\infty) \cap K_\infty) \backslash K_\infty = (P_1(\mathbb{A}_\infty) \cap K_\infty) \backslash K_\infty $. The function $\varphi$ was chosen to be $K$-finite, in particular it is $K_\infty$-finite. Suppose that $\varphi=\otimes_v \varphi_v$, and $\varphi_\infty=\otimes_{v|\infty} \varphi_v$. Since the integrand of \eqref{blah23} is a smooth function of $k$, in the region of absolute convergence, $A(-,s)$ is  smooth function. 

There is some choice of data so that $A(1,s_0) \neq 0$.
The function $\varphi_\infty$ is $K_\infty$-finite. Let $\varphi_\infty^\circ=\bigotimes_{v|\infty} \phi_v^\circ(X_v)$. According to \eqref{schrodinger} it is of the form
\begin{equation*}
 \varphi_\infty(X)= p(X)  \varphi_\infty^\circ(X)
\end{equation*}
where $X =\bigotimes_{v|\infty} X_v \in \mathbb{X}(\mathbb{A}_\infty)$ and $p(X)$ is a polynomial in $\mathbb{X}(\mathbb{A}_\infty)$. In particular, $| \det(X) |_\infty$ is a polynomial in $\mathbb{X}(\mathbb{A}_\infty)$. Pick $p(X)$ to be of the form
\begin{equation*}
 p(X)=q(X) \cdot |\det(X)|_\infty ^n
\end{equation*} 
where $q(X)$ is a polynomial in $\mathbb{X}(\mathbb{A}_\infty)$ and $n \in \mathbb{N}$. By Lemma~\ref{boundedlemma} $\phi^{T, \nu}$ is bounded, so in particular it is bounded on $M_1(\mathbb{A}_\infty)$. Therefore,
\begin{align}
 A(1,s_0) = \int \limits_{\text{GL}_2(\mathbb{A}_\infty)} & |\det(a)|_\infty^{s_0-1+n} \, q(a) \phi^{T, \nu}(m(a)) \, da. \label{integralA}
\end{align}
For $Re(s_0-1+n)>>0$ the integral converges absolutely. Indeed, $\varphi_\infty^\circ$ decays exponentially as the entries of $a$ become large while the rest of the integrand has polynomial growth at infinity. As the entries of $a$ become small, so does $|\det(a)|^{s_0-1+n}$. The other terms in the integrand are bounded.

By assumption on $\phi$, there is some $g \in G_{1}(\mathbb{A}_\infty)$ so that $\phi^{T,\nu}(g) \neq 0$. By the Iwasawa decomposition I can write $g_\infty=na^\prime k$, where $n \in N(\mathbb{A}_\infty)$, $a^\prime \in M_1(\mathbb{A}_\infty)$, and $k \in K_{1, \infty}$. Since $K_{1, \infty}$ acts on the space of $\pi$, replace $\phi$ with $\pi(k) \phi$ because the action of $\pi$ is compatible with taking Bessel coefficients. Assume $k=1$. Since $\phi^{T,\nu}(n a^\prime)=\psi_T(n) \cdot \phi^{T,\nu}(a^\prime)$, and $\psi_T(n) \neq 0$, then $\phi(a^\prime)\neq 0$.
Since polynomials are dense in $L^2(\mathbb{X}(\mathbb{A}_\infty))$, then there is some choice of polynomial $q$ so that $A(1,s_0) \neq 0$.

Therefore, $A(1,s)$ is a nonzero holomorphic function in a neighborhood of $s_0$, and $A(k,s)$ is a $K$-finite function of $k$ on $(M_1(\mathbb{A}_\infty) \cap K_\infty) \backslash K_\infty$.
There is a bijection between $\bigotimes \limits_{v|\infty} Ind_{P(F_v)}^{G(F_v)}$ and $K_v$ finite functions in $L^2((M_1(\mathbb{A}_\infty \cap K_\infty) \backslash K_\infty)$ given by restricting $f$ to $K_v$. 
Therefore, there is a choice of $K$-finite standard section $f(k)$ so that 
\begin{equation*}
 \int \limits_{K_\infty} f(k,s_0) \, A(k,s_0) \, dk \neq 0. \qedhere
\end{equation*}
\end{proof}

\section{Proof of Theorem 1}
This section summarizes the results of previous sections to prove Theorem \ref{maintheorem} which is restated here.

\begin{restatement}
 Let $\pi$ be a cuspidal automorphic representation of GSp$_4(\mathbb{A})$, and $\phi=\otimes_v \phi_v \in V_\pi$ be a decomposable automorphic form. Let $T$ and $\nu$ be such that $\phi^{T,\nu}\neq0$. There exists a choice of section $f(s, -) \in \text{Ind}_{P(\mathbb{A})}^{G(\mathbb{A})}( \delta_P ^{1/3(s-1/2)})$, and some $\varphi=\otimes_v \varphi_v \in \mathcal{S}( \mathbb{X}(\mathbb{A}))$ such that the normalized integral
\begin{equation*}
 I^*(s;f, \phi, T, \nu, \varphi)= d(s) \cdot L^{S}(s, \pi \otimes \chi_T)
\end{equation*}
where $S$ is a finite set of bad places including all the archimedean places. Furthermore, for any complex number $s_0$, the data may be chosen so that $d(s)$ is holomorphic at $s_0$, and $d(s_0) \neq 0$.
\end{restatement}
\begin{proof}
There is a finite set of places $S$ including all the archimedean places, such that for $v \notin S$, the conditions in Definition \ref{unramifieddef} are satisfied. Consider the normalized Eisenstein series $E^*(s,f,g)=\zeta^S(s+1) \, \zeta^S (2s) E(s,f,g)$ that was described in \eqref{normalizing}. Define $I^*(s;f, \phi, T, \nu, \varphi)$ to be the global integral defined in \eqref{10} except that $E(s,f,g)$ is replaced by $E^*(s,f,g)$.

By Proposition \ref{eulerproduct} the integral factors as 
\begin{equation*}
 I^*(s;f, \phi, T, \nu, \varphi)=I_\infty(s) \times \prod \limits_{\substack{v < \infty\\ v \in S}} I_v(s) \times \prod \limits_{ v \notin S} I^*_v(s).
\end{equation*}
Here, $I^*_v(s)= \zeta_v (s+1) \zeta_v (2s) I_v(s)$.
According to Proposition \ref{inertprop1} and Proposition \ref{splitprop1} for $v \notin S$,
\begin{equation*}
I^*_v(s)=L(s, \pi_v \otimes \chi_{T,v}).
\end{equation*}
By Proposition \ref{finiteram} for every finite place $v \in S$, there is a choice of local section $f_v(s, -)$ and local Schwartz-Bruhat function $\varphi_v$ so that
\begin{equation*}
 I_v(s) \equiv 1.
\end{equation*}
By Proposition \ref{archprop} there is a choice of data at the infinite places, $f_\infty(s, -)$, and $\varphi_\infty$, so that $I_\infty(s)$ is holomorphic at $s_0$, and $I_\infty(s_0) \neq 0$. 

Choose $f(s, -)=\otimes_v f_v(s, -)$ so that $f_v(s,-)$ is the choice specified above for $v \in S$, and the local spherical section otherwise. Similarly, choose $\varphi = \otimes_v \varphi_v$ so that $\varphi_v$ is the choice specified for $v \in S$, and the local spherical Schwartz-Bruhat function otherwise. This completes the proof of the theorem.
\end{proof}

\end{document}